\documentclass[10pt]{article}
\usepackage{amsfonts, epsfig, amsmath, amssymb, color,amsthm}
\usepackage{textcomp}
\usepackage{paralist}
\usepackage[normalem]{ulem}

\usepackage[english]{babel}
\newcommand{\BB}{\mathbb{B}}
\newcommand{\CC}{\mathbb{C}}
\newcommand{\DD}{\mathbb{D}}
\newcommand{\EE}{\mathbb{E}}
\newcommand{\FF}{\mathbb{F}}

\newcommand{\II}{\mathbb{I}}
\newcommand{\JJ}{\mathbb{J}}
\newcommand{\KK}{\mathbb{K}}

\newcommand{\MM}{\mathbb{M}}
\newcommand{\NN}{\mathbb{N}}

\newcommand{\PP}{\mathbb{P}}
\newcommand{\QQ}{\mathbb{Q}}
\newcommand{\RR}{\mathbb{R}}
\newcommand{\SSS}{\mathbb{S}}

\newcommand{\XX}{\mathbb{X}}

\newcommand{\aA}{\mathcal{A}}
\newcommand{\bB}{\mathcal{B}}
\newcommand{\cC}{\mathcal{C}}

\newcommand{\eE}{\mathcal{E}}
\newcommand{\fF}{\mathcal{F}}
\newcommand{\gG}{\mathcal{G}}

\newcommand{\lL}{\mathcal{L}}

\newcommand{\pP}{\mathcal{P}}

\newcommand{\rR}{\mathcal{R}}
\newcommand{\sS}{\mathcal{S}}

\newcommand{\uU}{\mathcal{U}}

\newcommand{\xX}{\mathcal{X}}
\newcommand{\yY}{\mathcal{Y}}

\newcommand{\vt}{\vartheta}
\newcommand{\e}{\varepsilon}

\newcommand{\no}{\noindent}

\newcommand{\ra}{\rightarrow}

\newcommand{\lra}{\longrightarrow}

\newcommand{\non}{\nonumber}

\newcommand{\lqq}{\leqslant}
\newcommand{\gqq}{\geqslant}

\newtheorem{thm}{Theorem}[section]

\theoremstyle{plain} 
\newtheorem{theorem}{Theorem}[section]
\newtheorem{corollary}{Corollary}[section] 
\newtheorem{lemma}{Lemma}[section] 
\newtheorem{proposition}{Proposition}[section] 

\theoremstyle{definition} 
\newtheorem{definition}{Definition}[section]

\newtheorem{condition}{Hypothesis}%[section]

\theoremstyle{remark} 
\newtheorem{remark}{Remark}[section]
\theoremstyle{definition}

\newtheorem{rem}[thm]{Remark}

\usepackage{lineno}
\DeclareMathSymbol{\ophi}{\mathalpha}{letters}{"1E}

\renewcommand{\phi}{\varphi}

\newcommand{\be}{\begin{equation}}
\newcommand{\ee}{\end{equation}}
\newcommand{\ben}{\begin{equation*}}
\newcommand{\een}{\end{equation*}}

\newcommand{\ba}{\begin{equation}\begin{aligned}}
\newcommand{\ea}{\end{aligned}\end{equation}}

\newfont{\cyrfnt}{wncyr10}
\def\J3{\cyrfnt{\rm \u{\cyrfnt I}}}
\def\j3{\cyrfnt{\rm \u{\cyrfnt i}}}

\usepackage[]{color}
\definecolor{DarkGreen}{rgb}{0.1,0.7,0.3}   %define a custom color

\definecolor{DarkGreen}{rgb}{0.1,0.7,0.3}   %define a custom color

\allowdisplaybreaks[4]

\begin{document}
\title{Large deviations for L\'{e}vy diffusions in the small noise regime 
}

\date{\null}

\author{
Pedro Catuogno \footnote{Departamento de Matem\'{a}tica Universidade Estadual de Campinas 13081-970 Campinas SP-Brazil; 
pedrojc@unicamp.br
} \quad \quad \quad \quad \quad 
Andr\'e de Oliveira Gomes \footnote{Departamento de Matem\'{a}tica Universidade Estadual de Campinas 13081-970 Campinas SP-Brazil; ENSTA-ParisTech Applied Mathematics Department,
828 Boulevard des Maréchaux, 91120 Palaiseau, France; andre.deoliveiragomes@cardis.io}  \hspace{2cm}
}

\maketitle

 \begin{abstract} 
This article concerns the large deviations regime and the consequent solution of the Kramers problem for a two-time scale stochastic system driven by a common jump noise signal perturbed in small intensity $\varepsilon>0$ and with accelerated jumps by intensity $\frac{1}{\varepsilon}$.  We establish\textit{ Freidlin-Wentzell estimates} for the slow process of the multiscale system in the small noise limit $\varepsilon \ra 0$ using the weak convergence approach to large deviations theory. The core of our proof is the reduction of the large deviations principle to the establishment of a stochastic averaging principle for auxiliary controlled processes. As consequence we solve the first exit time/ exit locus problem from a bounded domain containing the stable state of the averaged dynamics for the family of the slow processes in the small noise limit.   
 \end{abstract} 

%\tableofcontents
%\newpage

\section{Introduction}

\no A large deviations principle (LDP for short) is a refinement of the Law of Large Numbers in the sense it encodes a much finer asymptotic analysis concerning the exponential decay of probabilities of unlikely events with respect to a certain parameter and exhibiting the rate of decay in terms of a deterministic functional that is denoted commonly in the literature as good rate function. Historically large deviations theory (LDT) made its first appearence in 1877 \cite{Boltzmann, Ellis99} in the context of Boltzmann's studies of the second law of thermodynamics. Other landmark on the field was given by the seminal result by Cr\'{a}mer in the context of risk probabilities and rapidly this mathematical branch evolved with a diversity of applications and interactions with other fields especially after the grounbreaking contributions from Varadhan \cite{Rama, Varadhan66, DV75, Varadhan, Varadhan08}. As excellents and exemplary monographs on the field we refer \cite{DZ98,  DeuStr89, Hollander}. With physical examples in mind large deviations principles can refer to extreme events such as systems that exchange from one equilibrium state to another and that occur with small probability. We refer to \cite{FW98, Ellis85, Olivieri Vares} where applications of large deviations regimes to statistical mechanics and to the study of metastable systems are respectively developed.

\no In this work we import several techniques from LDT for Poisson random measures in order to understand asymptotically, as the source of noise vanishes, qualitative features concerning  the following multiscale stochastic system driven by a pure jump noise signal. Fixed $T>0$ and $\varepsilon>0$ let
\begin{align} \label{eq: multiscale-introduction}
\begin{cases}
d X^\varepsilon_t &= a(X^\varepsilon_t, Y^\varepsilon_t)dt + \varepsilon  \displaystyle \int_{\RR^d \backslash \{0\}} c(X^\varepsilon_{t-},z) d \tilde N^{\frac{1}{\varepsilon}}(ds); \\
d Y^\varepsilon_t &= \displaystyle \frac{1}{\varepsilon} f(X^\varepsilon_t, Y^\varepsilon_t) ds + \int_{\RR^d \backslash \{0\}} h(X^\varepsilon_{t-}, Y^\varepsilon_{t-},z) d \tilde N^{\frac{1}{\varepsilon}}(dz); \quad t \in [0,T]; \\
X^\varepsilon_0 & = x \in \RR^d; \quad Y^\varepsilon_0=y \in \RR^k.
\end{cases}
\end{align}

\no For every $\varepsilon>0$ the stochastic process $(X^\varepsilon_t, Y^\varepsilon_t)_{t \in [0,T]}$ takes values in $\RR^n := \RR^d \times \RR^k$ and the driving signal $\tilde N^{\frac{1}{\varepsilon}}$ is a compensated Poisson random measure with intensity $\frac{1}{\varepsilon} \nu \otimes ds$, where $ds$ stands for the Lebesgue measure on the real line and $\nu$ is a L\'{e}vy measure on $(\RR^d \backslash \{0\}, \bB(\RR^d \backslash \{0\}))$. We consider $\nu$ possibly with infinite total mass but satisfying an exponential integrability assumption that reads as the big jumps of the underlying L\'{e}vy process that drives (\ref{eq: multiscale-introduction}) having exponential moments of order 2. The assumptions on the coefficients of (\ref{eq: multiscale-introduction}) and on the measure $\nu$ will be precised with full detail in the following section.

\no Multi-scale systems as (\ref{eq: multiscale-introduction}) are nowadays very trendy in mathematical and physical disciplines since they succesfully capture phenomena that exhibit different levels of heterogeneity/homogeneity categorized by a scaling parameter. Typical examples are multi-factor stochastic volatility models in Finance \cite{Papanicolaou1, Papanicolaou2} and dynamics of proxy-data in Climatology \cite{Kifer}. Usually such systems are highly complex and difficult to simulate \cite{Liu and Vanden-Eijden}. Due to that complexity and the goal of approximating the dynamics of such systems by much simpler ones the idea of stochastic averaging was introduced by Khasminkki in \cite{Khasminskii} as follows. Under strong dissipativity assumptions on the coefficients of the fast variable $(Y^\varepsilon_t)_{t \in [0,T]}$ of (\ref{eq: multiscale-introduction}) that ensure the existence of a unique invariant measure $\mu^x$ for the fast variable process with frozen variable $x$ and such that a certain mixing ergodic property holds for the averaged mixing coefficient
\begin{align} \label{eq: mixing coefficient-int}
\bar a(x)&:= \int_{\RR^k} a(x,y) \mu^x(dy),
\end{align} 
the strong averaging principle says that,  for every $T>0$ and $\delta>0$,  the following convergence holds
\begin{align} \label{eq: strong averaging principle-int}
\displaystyle \lim_{\varepsilon \ra 0} \PP \Big ( \displaystyle \sup_{t \in [0,T]} |X^\varepsilon_t - \bar X^0_t| > \delta \Big )=0,
\end{align}
where $\bar X^0$ is the unique solution of the ordinary differential equation
\begin{align} \label{eq: int- averaged deterministic system}
\begin{cases}
\frac{d}{dt} \bar X^0_t &= \bar a(\bar X^0_t), \quad t \in [0,T]; \\
\bar X^0_0&=x.
\end{cases}
\end{align}

\no The Khasminkii's technique was introduced in \cite{Khasminskii} and later implemented by Mark Freidlin \cite{Freidlin78} and Veterennikov in \cite{Vetennikov} in different contexts. As examples of works on stochastic averaging principles we mention \cite{Cerrai, Cerrai2, Cerrai4} concerning multi-scale stochastic partial differential equations (SPDEs for short) driven by Gaussian signals and \cite{Givon, Liu, Xu-Miao-Liu, Xu} for multi-scale systems driven by jump diffusions.

\no Although the averaging principle (\ref{eq: strong averaging principle-int}) gives an approximation result for small $\varepsilon>0$ of the slow variable process $(X^\varepsilon_t)_{t \in [0,T]}$ by the averaged dynamics of $\bar X^0$ nothing is said on the rate of convergence. A large deviations principle provide sharper estimates within the identification of the rate of convergence for (\ref{eq: strong averaging principle-int}) in an exponentially small scale $\varepsilon \ra 0$ in terms of the good rate function. We refer to \cite{BDG18, Duan12, Kumar17, Vetennikov2} as examples of stochastic averaging principles under the large deviations regime.

\no The first goal of this work is to derive the large deviations principle (LDP for short) for $(X^\varepsilon)_{\varepsilon>0}$ given by (\ref{eq: multiscale-introduction}). We show that $(X^\varepsilon)_{\varepsilon>0}$ satisfies a large deviations principle in $\DD([0,T];\RR^d)$ the space of c\`{a}dl\'{a}g functions endowed with the Skorokhod topology (Section 12 in \cite{Billingsley}) and the good rate function $\JJ: \DD([0,T]; \RR^d) \longrightarrow [0, \infty]$ given by
\begin{align*}
\II(\psi) := \displaystyle \inf_{g \in \mathbb{S}} \int_0^T \int_{\RR^d \backslash \{0\}} (  g(s,z) \ln g(s,z) - g(s,z) + 1) \nu(dz) ds
\end{align*}
where 
\begin{align*}
\mathbb{S}:= \bigcup_{M>0} S^M := \bigcup_{M>0} \Big  \{ &g: [0,T] \times \RR^d \backslash \{0\} \longrightarrow [0,\infty) ~ |~ \int_0^T \int_{\RR^d \backslash \{0\}} ( g(s,z) \ln g(s,z) - g(s,z) + 1) \nu(dz) ds \leq M \Big \}. 
\end{align*}
and  for every $x \in \RR^d$ and $g \in \mathbb{S}$ the function $\psi \in \DD([0,T]; \RR^d)$ solves uniquely the skeleton equation
\begin{align} \label{eq: the skeleton eq-int}
\psi(t)=x+  \int_0^t \bar a (\psi(s))ds + \int_0^t \int_{\RR^d \backslash \{0\}} c(\bar X^0_s,z) (g(s,z)-1) \nu(dz)ds, \quad t \in [0,T].
\end{align}
\no This means that $\JJ$ has compact sublevel sets in the Skorokhod topology $\DD([0,T];\RR^d)$ and that for every open set $G \in \bB(\DD([0,T]; \RR^d))$ and closed set $F \in \bB(\DD([0,T]; \RR^d))$ the following holds
\begin{align} \label{eq: int-LDP}
\displaystyle \liminf_{\varepsilon \ra 0} \varepsilon \ln \PP(X^\varepsilon \in F) & \geq - \displaystyle \inf_{\psi \in F} \JJ(\psi) \quad \text{and} \nonumber \\
\displaystyle \limsup_{\varepsilon \ra 0} \varepsilon \ln \PP(X^\varepsilon \in G)  & \leq - \displaystyle \inf_{\psi \in G} \JJ(\psi).
\end{align}
\no The second goal of this work is, under stricter assumptions on the measure $\nu$, to use the LDP (\ref{eq: int-LDP}) in order to solve the asymptotic behavior $\varepsilon \ra 0$ of the law and the expectation of the first exit time and location (known as Kramers problem in the literature)
\begin{align} \label{eq: int-first exit time}
\sigma^\varepsilon(x):= \inf \{ t \geq 0 ~|~ X^{\varepsilon,x}_t \notin D \} \quad \text{ and } X^{\varepsilon,x}_{\sigma^\varepsilon(x)} \quad \text{respectively},
\end{align} 
where, for every $\varepsilon>0$, we stress the dependence on the initial condition $x \in D$ that lives on a neighborhood $D \subset \RR^d$ of $0$ which we assume to be a stable state of the averaged dynamical system (\ref{eq: int- averaged deterministic system}). The Kramers problem arose firstly in the context of chemical reaction kinetics \cite{Arrhenius-89, Eyring-35, Kramers-40}. Nowadays this is a classical problem in Probability theory and provides crucial insights in many areas ranging from statistical mechanics, statistics, risk analysis, population dynamics, fuid dynamics to neurology. It was also a driving force in the development of large deviations theory in the small noise limit for Gaussian dynamical systems in many different settings and effects derived such as metastability and stochastic resonance. Classical texts with detailed exposition include \cite{Berglund-13, BerglundG-04, BerGen-10, BarBovMel10, BovEckGayKle02, Bovier1, CerRoeck-04, Freidlin00, Galves Vares, Siegert} and examples of the recent active research on the field include \cite{19-3, Gamboa, 19-4, Locherbach, 19-1}. The establishment of large deviations principles in the small noise limit and its use to solve the Kramers problem and the metastable behavior of the perturbed dynamical systems by Gaussian signals is nowadays commonly designated by \textit{Freidlin-Wentzell theory}.

\no Nevertheless,  the literature on large deviations theory and the application to the study of the Kramers problem for Markovian systems with jumps is more fragmented and recent. One reason for the ausence of a \textit{Freidlin-Wentzell theory} for dynamical systems perturbed in low intensity by jump signals is the variety of L\'{e}vy processes, including processes with heavy tails and the resulting lack of moments. LDPs for certain classes of L\'{e}vy noises and Poisson random measures are given in \cite{Acosta, Blanchet et al17, Borovkov, Florens Pham, Godovanchuk-82, Leonard, Lynch87, Puhalskii}. The first exit time problem for small jump processes starts with the seminal work \cite{Godovanchuk-79} for $\alpha$-stable processes and for more general heavy-tailed processes by \cite{Debussche et al., HoePav-13, Imk Pav06, Imk Pav08, Pav11}. The above mentioned works do not follow though a large deviations regime since the intensity of the jumps is not rescaled in a inverse way by $\frac{1}{\varepsilon}$ from the signal strength $\varepsilon>0$ affecting the size of the jumps.  It is this tuning between the jump size $\varepsilon>0$ and the intensity measure of the Poisson random measure $N^{\frac{1}{\varepsilon}}$ that permits to retrieve the large deviations regime for dynamical systems perturbed by small jump noises,  as it is exampled in the Appendix for a simple but sufficiently rich toymodel with values in $\RR^d$ given for any $\varepsilon>0$ and $t \in [0,T]$, $T>0$ by $L^\varepsilon_t:= \varepsilon \int_0^t \int_{\RR^d \backslash \{0\}} z N^{\frac{1}{\varepsilon}}$,  where the underlying Poisson random measure has intensity given by $\frac{1}{\varepsilon} \nu \otimes ds$ for the explicit L\'{e}vy measure $\nu(dz)= \frac{1}{|z|^{\beta+d}} e^{- |z|^\alpha}dz \quad \alpha >1, \beta \in [0,1).$

\no Our method to prove the large deviations principle for the family of slow variables $(X^\varepsilon)_{\varepsilon>0}$ in (\ref{eq: multiscale-introduction}) relies on the weak convergence approach of Dupuis, Ellis, Budhiraja and collaborators, specifically on the works \cite{BDM11, BCD13}. The weak convergence approach to LDT builds up on the equivalence in Polish spaces between the definition of LDP and the variational principle known in the literature as the \textit{Laplace-Varadhan principle}.  Initially Fleming applied in \cite{Fleming1, Fleming3} methods of stochastic control to LDT. The control-theoretical approach was carried out later in order to derive variational formulas for Laplace functionals of Markov processes in different contexts (cf. \cite{Dupuis Ellis}). In \cite{BD00} the authors derive a sufficient abstract condition for Brownian diffusions and later for jump-diffusions in \cite{BDM11, BCD13} through the establishment of variational formulas for Laplace functionals of Markov provesses. We refer the reader to the recent book \cite{Budhiraja book} for an out-to-date introduction to the subject. 

\paragraph*{Comments on the results.} 

\no The proof of the main result of this work follows from an abstract sufficient condition for large deviations principles stated as Theorem 4.2 in \cite{BDM11}.  In our case the application of this abstract condition is not straightforward due to the coupling between the slow variable $X^\varepsilon$ and the fast variable $Y^\varepsilon$ in (\ref{eq: multiscale-introduction}) with different scaling orders in $\varepsilon \ra 0$.  

\no More precisely the difficult part is to prove directly the following. Fix $M>0$. For every $\varepsilon>0$ let $\varphi^\varepsilon \in S^M$ random such that $\varphi^\varepsilon \Rightarrow \varphi$ in law as $\varepsilon \ra 0$. For every $\varepsilon>0$, $T>0$, $x \in \RR^d$, $y \in \RR^k$ and $t \in [0,T]$ let $(\xX^\varepsilon_t)_{t \in [0,T]}$ and $(\yY^\varepsilon_t)_{t \in [0,T]}$ be the controlled processes given by the following controlled SDEs:
\begin{align} \label{eq: Khasminki- controlled X variable}
 \xX^\varepsilon_t &= x + \displaystyle \int_0^t \Big ( a(\xX^\varepsilon_s, \yY^\varepsilon_s) + \int_{\RR^d \backslash \{0\}} c(\xX^\varepsilon_s,z) (\varphi^\varepsilon(s,z)-1) \nu(dz)\Big ) ds + \varepsilon \int_0^t \int_{\RR^d \backslash \{0\}} c(\xX^\varepsilon_{s-},z) \tilde N^{\frac{1}{\varepsilon} \varphi^\varepsilon}(ds,dz)
\end{align}
and
\begin{align} \label{eq: Khasminkii- controlled Y variable}
\yY^\varepsilon_t &=y + \displaystyle \frac{1}{\varepsilon}\int_0^t \Big (f(\xX^\varepsilon_s, \yY^\varepsilon_s) + \int_{\RR^d \backslash \{0\}} h(\xX^\varepsilon_s, \yY^\varepsilon_s,z) (\varphi^\varepsilon(s,z)-1) \nu(dz) \Big ) ds + \int_0^t \int_{\RR^d \backslash \{0\}} h(\xX^\varepsilon_{s-}, \yY^\varepsilon_{s-},z) \tilde N^{\frac{1}{\varepsilon} \varphi^\varepsilon}(ds,dz)
\end{align}
where,  for any $\varepsilon>0$,  the random measure $\tilde N^{\frac{1}{\varepsilon} \varphi^\varepsilon}$ is a controlled random measure that under a change of probability measure has the same law of $\tilde N^{\frac{1}{\varepsilon}}$ under the original probability measure. This will be rigorously stated in Section  \ref{section: proof}. \\

\no Under the following setting, the main task in the derivation of the LDP is to prove that $\xX^\varepsilon \Rightarrow \bar \xX$ where $\bar \xX$ solves (\ref{eq: the skeleton eq-int})uniquely in $C([0,T]; \RR^d)$ for the control $\varphi \in S^M$. In order to prove that convergence in law we show that the family $(\xX^\varepsilon)_{\varepsilon>0}$ satisfies a tightened averaging principle,  i.e.  for every $\delta>0$ the following holds
\begin{align} \label{eq: the controlled averaging principle int}
\displaystyle \limsup_{\varepsilon \ra 0} \PP \Big ( \displaystyle \sup_{t \in [0,T]} |\xX^\varepsilon(t) - \bar \xX^\varepsilon(t)|> \delta  \Big )=0,
\end{align}
where $(\bar \xX^\varepsilon)_{\varepsilon>0}$ is defined for every $\varepsilon>0$ and $t \in [0,T]$ by
\begin{align} \label{eq: Khasminkii- controlled averaged variable}
\bar \xX^\varepsilon_t &= x+ \displaystyle\int_0^t \Big ( \bar a(\bar \xX^\varepsilon_s)+ \int_{\RR^d \backslash \{0\}} c(\bar \xX^\varepsilon_s,z) (\varphi^\varepsilon(s,z)-1) \nu(dz) \Big ) ds  + \varepsilon \int_0^t \int_{\RR^d \backslash \{0\}} c(\bar \xX^\varepsilon_{s-},z) \tilde N^{\frac{1}{\varepsilon} \varphi^\varepsilon}(ds,dz).
\end{align}
This will imply by Slutzky's theorem (Theorem 4.1 in \cite{Billingsley}) that $(\xX^\varepsilon)_{\varepsilon>0}$ has the same weak limit of $(\bar \xX^\varepsilon)_{\varepsilon>0}$. And therefore we are conducted to the (easier) task to show that $\bar \xX^\varepsilon \Rightarrow \bar \xX$ (since the dynamics of (\ref{eq: Khasminkii- controlled averaged variable}) is decoupled from the dynamics of the fast variable of the original stochastic system (\ref{eq: multiscale-introduction})). \\

\no The proof of the tightened controlled averaging principle (\ref{eq: the controlled averaging principle int}) is inspired on the classical Khasminkii's technique introduced in \cite{Khasminskii}. In a nutshell the procedure relies on a discretization of the time interval $[0,T]$ in a finite number of intervals with same length $\Delta(\varepsilon) \ra 0$ as $\varepsilon \ra 0$ satisfying some growth conditions that will interplay with the ergodic properties of the averaged dynamics via the construction of auxiliary processes $(\hat \xX^\varepsilon)_{\varepsilon>0}$ and $(\hat \yY^\varepsilon)_{\varepsilon>0}$. 

\no Our main result shows in particular that $(X^\varepsilon)_{\varepsilon>0}$ obeys the same LDP of $(\bar X^\varepsilon)_{\varepsilon>0}$ where we define the averaged process $\bar X^\varepsilon$ for every $\varepsilon>0$ and $t \in [0,T]$ by
\begin{align} \label{eq: averaged processes}
\bar X^\varepsilon_t= x + \int_0^t \bar a( \bar X^\varepsilon_s) ds + \varepsilon \int_0^t \int_{\RR^d \backslash \{0\}} c(\bar X^\varepsilon_{s-},z) \tilde N^{\frac{1}{\varepsilon}}(ds,dz).
\end{align}
\no One could firstly derive the LDP for $(\bar X^\varepsilon)_{\varepsilon>0}$ and secondly show that the families $(X^\varepsilon)_{\varepsilon>0}$ and $(\bar X^\varepsilon)_{\varepsilon>0}$ are exponentially equivalent, i.e. for every $\delta>0$ we have
\begin{align} \label{eq: int-exponential negligibility}
\displaystyle \lim_{\varepsilon \ra 0} \varepsilon \ln \PP \Big ( \sup_{0 \leq t \leq T} | X^\varepsilon_t - \bar X^\varepsilon_t| > \delta\Big )=-\infty.
\end{align}
This would imply that $(X^\varepsilon)_{\varepsilon>0}$ obeys the same LDP of $(\bar X^\varepsilon)_{\varepsilon>0}$ as $\varepsilon \ra 0$. However verifying the exponential equivalence of those families is in general hard. The reasoning employed in this work illustrates the robustness of the weak convergence approach providing a way of reducing the proof of the LDP to the verification of properties concerning continuity and tightness of certain auxiliary processes associated to $(X^\varepsilon)_{\varepsilon>0}$. Such reduction of complexity in such endeavour can be appreciated immediately by the contrast between the zero scale of the limit (\ref{eq: strong averaging principle-int}) with the exponential negligibility demanded in the establishment of the limit (\ref{eq: int-exponential negligibility}).

\no Due to the exponential equivalence of $(X^\varepsilon)_{\varepsilon>0}$ and $(\bar X^\varepsilon)_{\varepsilon>0}$ the LDP of $(X^\varepsilon)_{\varepsilon>0}$ is written with the good rate function of the LDP of $(\bar X^\varepsilon)_{\varepsilon>0}$ which has important implications. It follows that the LDP of $(X^{\varepsilon,x})_{\varepsilon>0}$ is given as an optimization problem under the averaged dynamics solved by continuous controlled paths with a nonlocal component due to the pure jump noise. Therefore in the solution of the Kramers problem for $(X^\varepsilon)_{\varepsilon>0}$ we are therefore allowed to write the potential height for the description of the law and the expected first exit time of $(X^\varepsilon)_{\varepsilon>0}$ from the domain $D \subset \RR^d$ containing the stable state of the averaged dynamics (\ref{eq: int-first exit time}). 

\no The solution of the Kramers problem for $(X^\varepsilon)_{\varepsilon>0}$ follows as this point analogously to what was done by the authors in the work \cite{Hogele AO}. Analogously to the classical Freidlin-Wentzell theory we solve the Kramers problem with a  pseudo-potential given in terms of the good rate function of the LDP.  In the Brownian 
case, under very mild assumptions on the coefficients of the SDE the respective controlled dynamics exhibits continuity properties that are crucial in the characterization of the first exit times. This differs  strongly from the pure jump case.  In this context, obtaining a closed form for the rate function is a hard task since the class of minimizers are scalar functions that 
represent shifts of the compensator of $\varepsilon \tilde N^{\frac{1}{\varepsilon}}$ on the 
nonlocal (possibly singular) component of the underlying controlled dynamics. This is an additional  
difficulty in the characterization of the first exit time in terms of the pseudo-potential. However, 
in the case of finite and symmetric jump measures we can solve the first exit time problem with 
the help of explicit formulas that we obtain for the controls. In other words, on an abstract 
level the physical intuition 
remains intact; however, since the control is given as a density w.r.t. 
the L\'evy measure $\nu$,
it is often hard to calculate the energy minimizing paths. 

\no Analogously to the Brownian case \cite{DZ98, FW98} we construct for the lower bound of the first exit time a (modified) Markov chain approximation that takes into account the topological particularities of 
the Skorokhod space on which we have the LDP. Here the 
symmetry of the measure $\nu$ plays an essential role since it enables us to derive the lower bound of the 
probabilities of exit in terms of probabilities of excursions from neighborhoods of the stable 
state of the deterministic dynamical system. Otherwise we would need to control the trajectories 
of perturbations from the deterministic dynamical system including the non-vanishing compensator 
of $\varepsilon N^{\frac{1}{\varepsilon}}$.

\no The proof of (\ref{thm: location exit}) in Theorem \ref{thm: exit time} follows the lines of the proof for the Brownian case but rests on several auxiliary results which derivation take into account the specific large deviations principle for $(X^\varepsilon)_{\varepsilon>0}$, particularities of the Skorokhod topology, the reduction of the dynamics of $(X^\varepsilon)_{\varepsilon>0}$ from the bounded domain $D$  to estimates on excursions from certain balls contained in the $D$ and  the exponential equivalence between $(X^\varepsilon)_{\varepsilon>0}$ and $(\bar X^\varepsilon)_{\varepsilon>0}$ that permits to describe the potential for the solution of the Kramers problem as the potential associated to the averaged (and simplified) dynamics of $(\bar X^\varepsilon)_{\varepsilon>0}$.

\paragraph*{Notation.} The arrow $\Rightarrow$ means convergence in distribution. Throughout the article we use when convenient the shorthand notation $A (\varepsilon)\lesssim_\varepsilon B(\varepsilon)$ to mean  that there exist a constant $ c>0$ independent of $\varepsilon>0$ and $\varepsilon_0>0$ such that $A(\varepsilon) \leq cB(\varepsilon)$ for every $\varepsilon< \varepsilon_0$. We write $A(\varepsilon) \simeq_\varepsilon B(\varepsilon)$ as $\varepsilon \ra 0$ to mean that $A(\varepsilon) \lesssim_\varepsilon B(\varepsilon)$ and $B(\varepsilon) \lesssim_\varepsilon A(\varepsilon)$ as $\varepsilon \ra 0$.

\section{The multi-scale system and statement of the main results}
\subsection{Notation and the probabilistic setting} \label{subsection: prob setting}
\no Let $T>0$. Let $\MM$ be the space of locally finite measures defined on $([0,T] \times \RR^d \backslash \{0\}, \bB([0,T] \times \RR^d \backslash \{0\}))$ and let us fix $\nu \in \MM$. Let $\PP$ be the unique probability measure defined on $(\MM, \bB(\MM))$ such that the canonical map $N$ turns out to be a Poisson random measure with intensity $\nu \otimes ds$ where $ds$ stands for the Lebesgue measure on $[0,T]$.  For every $\varepsilon>0$ let $N^{\frac{1}{\varepsilon}}$ be the Poisson random measure defined on $(\MM, \bB(\MM), \PP)$  with intensity $\frac{1}{\varepsilon} \nu(dz) \otimes ds$ and $\tilde N^{\frac{1}{\varepsilon}}$ its compensated version. 

\no In what follows we augment the probability space in order to register not only the instant and the size of the jumps given by the Poisson random measure $N$ but also their intensity.  Consider $[0,T] \times \RR^d \backslash \{0\} \times [0,\infty)$ and let $\bar \MM$ be the space of all locally finite measures defined on the Borel sets of that Cartesian product. There exists a unique probability measure $\bar \PP$ under which the canonical map
\begin{align*}
&\bar N: \bar \MM \longrightarrow \bar \MM \\
& \bar N(\bar m) = \bar m
\end{align*}
turns out to be a Poisson random measure with intensity given by $ds \otimes \nu \otimes dr$ where $dr$ stands for the Lebesgue measure on $[0,\infty)$. 

\no We remark that for every $\varepsilon>0$ the Poisson random measure $N^{\frac{1}{\varepsilon}}$ can be seen as a controlled random measure with respect to $\bar N$ in the following way:
\begin{align*}
N^{\frac{1}{\varepsilon}}([0,s] \times A)(\omega)= \int_0^s \int_A \int_0^\infty \textbf{1}_{[0, \frac{1}{\varepsilon}]}(r) \bar N(ds,dz,dr), \quad s \geq 0; A \in \bB(\RR^d \backslash \{0\}).
\end{align*}

\no Given $t \in [0,T]$,  let  
\begin{align*}
\fF_t := \sigma \Big ( \bar N(0,s] \times A) ~|~0 \leq s \leq t, \quad A \in \bB(\RR^d \backslash \{0\} \times [0,\infty)) \Big ).
\end{align*}
\no Let $\FF:=\{\bar \fF_t\}_{t \in [0,T]}$ be the completion of $(\fF_t)_{t \in [0,T]}$ under $\bar \PP$ and consider $\bar \pP$ the predictable $\sigma$-field on $[0,T] \times \bar \MM$ with respect to the filtration $\{ \bar \fF_t \}_{t \in [0,T]}$ on $(\bar \MM, \bB(\bar \MM), \bar \PP)$.

\subsection{Hypothesis on the coefficients}
\no Fix $T>0$, $n=d+k$ with $d,k \in \NN$ and let $\nu \in \MM$. \\

\no For every $T>0$ and $\varepsilon>0$ we consider the following system of stochastic differential equations
\begin{align} \label{eq: the multiscale system}
\begin{cases}
X^\varepsilon_t &=x + \displaystyle \int_0^t a(X^\varepsilon_s, Y^\varepsilon_s)ds + \varepsilon \int_0^t \int_{\RR^d \backslash \{0\}} c(X^\varepsilon_{s-},z) \tilde N^{\frac{1}{\varepsilon}}(ds,dz) \\
Y^\varepsilon_t &= y + \displaystyle \frac{1}{\varepsilon} \int_0^t f(X^\varepsilon_s, Y^\varepsilon_s)ds + \int_0^t \int_{\RR^d \backslash \{0\}} h(X^\varepsilon_{s-}, Y^{\varepsilon}_{s-}) \tilde N^{\frac{1}{\varepsilon}}(ds,dz), \quad t \in [0,T].
\end{cases}
\end{align}

\no In order to guarantee existence and uniqueness of solution for (\ref{eq: the multiscale system}) we assume that the coefficients are deterministic measurable functions $a:\RR^d \times \RR^n \longrightarrow \RR^d$, $c:\RR^d \times \RR^d \backslash \{0\} \longrightarrow \RR^d$, $f: \RR^d \times \RR^n \longrightarrow \RR^{n}$ and $h: \RR^d \times \RR^n \times \RR^d \backslash \{0\} \longrightarrow \RR^n$ satisfying the following.

\begin{condition} \label{condition: hy coefficients existence uniqueness sol}
\begin{enumerate}
\item There exists $L>0$ such that for every $x, \bar x \in \RR^d$ and $y, \bar y \in \RR^n$ the following holds
\begin{align} \label{eq: condition Lipschitz}
|a(x,y) - a(\bar x, \bar y)| & \leq L \Big ( |x - \bar x| + |y - \bar y| \Big );  \nonumber\\
\int_{\RR^d \backslash \{0\}}|c(x,z) - c(\bar x, z)| \nu(dz) & \leq L |x- \bar x|; \nonumber \\
|f(x,y) - f(\bar x, \bar y)| & \leq L \Big ( |x - \bar x| + |y - \bar y| \Big ); \nonumber \\
\int_{\RR^d \backslash \{0\}}|c(x,y,z) - c(\bar x, \bar y,z)| \nu(dz)& \leq L \Big ( |x- \bar x| + |y - \bar y| \Big );
\end{align}
\item The functions $c(0,z)$ and $h(0,0,z)$ are in $L^1(\nu)$.
\end{enumerate}
\end{condition}

\begin{remark} Hypothesis (\ref{condition: hy coefficients existence uniqueness sol}) implies obviously that the coefficients of (\ref{eq: the multiscale system}) have sublinear growth.
\end{remark}

\begin{definition} \label{definition: solution SDE}
Given $T>0$, $\varepsilon>0$, $x \in \RR^d$ and $y \in \RR^k$ we consider the stochastic basis $(\bar \MM, \bB(\bar \MM),\FF,  \bar \PP)$. A strong solution of (\ref{eq: the multiscale system}) is a stochastic process $(X^\varepsilon_t, Y^\varepsilon_t)_{t \in [0,T]}$ that is $\{\bar \fF_t\}_{t \in [0,T]}$-adapted and that solves (\ref{eq: the multiscale system}) $\bar \PP$-a.s
\end{definition}

\no Given $T>0$, $m \in \NN$ and $\FF=\{ \bar \fF_t \}_{t \in [0,T]}$ we define the space
\begin{align*}
\sS^2_{\FF} ([0,T]; \RR^m) := \Big \{ \varphi: \bar \MM \times [0,T] \longrightarrow \RR^k ~|~ &\varphi \text{ is } \bar \FF-\text{adapted with c\`{a}dl\`{a}g paths such that } \\
 & \bar \EE \Big [ \displaystyle \sup_{t \in [0,T]} |\varphi(t)|^2 \Big ] < \infty \Big \}.
\end{align*}

\no The existence and uniqueness of the solution process $(X^\varepsilon_t, Y^\varepsilon_t)_{t \in [0,T]} \in \sS^2_{\FF}([0,T]; \RR^d) \times \sS^2_{\FF}([0,T]; \RR^n)$ of (\ref{eq: the multiscale system}) in the sense of Definition \ref{definition: solution SDE} follows from Lemma V.2 and Theorem V.7 of \cite{Protter}. This is the content of the following result.

\begin{theorem} \label{theorem: existence uniqueness solution}
Fix $\nu \in \MM$, $T, \varepsilon>0$, $x \in \RR^d$ and $y \in \RR^k$. Let us assume that Hypothesis \ref{condition: hy coefficients existence uniqueness sol} hold. Then there exists a stochastic process 
$$(X^\varepsilon_t, Y^\varepsilon_t)_{t \in [0,T]}\in \sS^2_{\bar \FF}([0,T]; \RR^d) \times \sS^2_{\bar \FF}([0,T]; \RR^n)$$ that solves uniquely (\ref{eq: the multiscale system}) in the sense of Definition \ref{definition: solution SDE}.  
\end{theorem}

\subsection{The averaged dynamics}
\no We make the further boundedness dissipativity assumptions on the coefficients of (\ref{eq: the multiscale system}) that yield the existence and uniqueness of solution for the averaged dynamics given by (\ref{eq: int- averaged deterministic system}).
\begin{condition} \label{condition: dissipativity}
\begin{enumerate}
 \item The function $a$ satisfies $a(0,y)=0$ for any $y \in \RR^k$ and there exists $\Lambda>0$ such that
 \begin{align} \label{eq: g and h bounded}
  |h(x,y,z)| & \leq \Lambda |z|, \quad \text{for every } x \in \RR^d, y \in \RR^k, z \in \RR^d \backslash \{0\}.
 \end{align}
\item
There exist constants $\beta_1, \beta_2>0$, such that, for any $x, \tilde x \in \RR^d$, $y, \tilde y \times \RR^k$   one has
\begin{align} \label{eq: dissipativity of the mixing coefficient}
\langle a(x, y) - a(\tilde x, y) , x - \tilde x  \rangle \leq - \beta_1 |x - \bar x|^2;
\end{align}
\begin{align} \label{eq: dissipativity on the coefficients}
2 \langle y, f(x,y) \rangle + \int_{\RR^d \backslash \{0\}} |h(x,y,z)|^2 \nu(dz)  &\leq - \beta_1  |y|^2 + \beta_2 |x|^2;
\end{align}
\begin{align} \label{eq: dissipativity Lipschitz on the coefficients}
2 \langle y - \tilde y , f(x,y) - f(x,\tilde y) \rangle &+ \int_{\RR^d \backslash \{0\}} |h(x,y,z) - h(x, \tilde y,z)|^2 \nu(dz)  \leq - \beta_1 |y - \tilde y|^2 + \beta_2 |x|^2;
\end{align}
and
\begin{align} \label{eq: dissipativity- the last one required }
2 \langle y - \tilde y , f(x,y) - f(\tilde x,\tilde y) \rangle \leq - \beta_1 |y- \tilde y|^2 + \beta_2 |x - \tilde x|^2.
\end{align}
\end{enumerate}
\end{condition}
\no The following a-priori estimates are straightforward and we omit their proofs.
\begin{proposition} \label{prop: a priori bound slow variable}
Fix $T$ and $y \in \RR^k$. Let Hypothesis \ref{condition: hy coefficients existence uniqueness sol}-\ref{condition: dissipativity} hold for some $\nu \in \MM$ and $x \in \RR^d$. There exists a constant $C_1>0$ independent of $\varepsilon>0$ such that  for all $0 < \varepsilon < 1$ we have
\begin{align} \label{eq: a priori bound slow-fast variable}
\bar \EE \Big [ \displaystyle \sup_{0 \leq t  \leq T}  |X^\varepsilon(t)|^2 \Big ]  + \displaystyle \sup_{0 \leq t \leq T} \bar \EE \Big [ |Y^\varepsilon(t)|^2  \Big ] \leq C_1.
\end{align}
\end{proposition}
\no We consider the equation for the fast variable of (\ref{eq: the multiscale system}) whenever the slow component is frozen and given by  $x \in \RR^d$ in the regime $\varepsilon=1$, i.e. fix $y \in \RR^k$; for every $t \geq 0$ let
\begin{align} \label{eq: the SDE for the fast variable with frozen slow component}
Y^{x,y}_t = y + \int_0^t f(x, Y^{x,y}_s)ds + \int_0^t \int_{\RR^d \backslash \{0\}} h(x, Y^{x,y}_{s-},z) \tilde N^{\frac{1}{1}}(ds,dz).
\end{align}
\no We assume that Hypotheses  \ref{condition: hy coefficients existence uniqueness sol}-\ref{condition: dissipativity} hold. We follow closely \cite{Cerrai, Cerrai2} in the argumentation below.\\
Fixed $x \in \RR^d$ we define the transition semigroup on the space $\BB_b(\RR^k)$ of the bounded measurable functions associated with the jump difusion defined by the strong solution of (\ref{eq: the SDE for the fast variable with frozen slow component}) by
\begin{align} \label{eq: the transition semigroup for the fast variable}
P^\zeta_t f(y) := \bar \EE [f(Y^{\zeta,y}_t)], \quad t \geq 0, \quad y \in \RR^k.
\end{align}
\no In what follows we discuss the existence and uniqueness of an invariant measure for the family of linear operators $(P^\zeta_t)_{t \geq 0}$, i.e. a probability measure $\mu^\zeta \in \pP(\RR^k, \bB(\RR^k))$ such that
\begin{align} \label{eq: invariant measure}
\int_{\RR^k} P^\zeta_t f(y) \mu^\zeta(dy) = \int_{\RR^k} f(y) \mu^\zeta(dy), \quad t \geq 0, f \in \BB_b(\RR^k)
.\end{align}
\no The dissipativity assumption given in (\ref{eq: dissipativity Lipschitz on the coefficients}) yields some $\cC>0$ such that, for any $T_0 \geq 0$, the following bound holds:
\begin{align} \label{eq: uniform bound for tightness}
\displaystyle \sup_{T \geq T_0} \bar \EE[|Y^{\zeta,y}(T)|^2] \leq \cC  e^{- 2 \beta_1 T} (1 + ||\zeta||_\infty^2 + |y|^2).
\end{align}  
\no The estimate (\ref{eq: uniform bound for tightness}) implies that the family of the laws of the process $\{ \lL(Y^{\zeta,y}(T))\}_{T \geq T_0}$ is tight in $\pP(\RR^k; \bB(\RR^k))$ when $T_0 \ra \infty$.  Prokhorov's theorem (Section 5 in \cite{Billingsley}) implies the existence of a weak limit $\mu^\zeta$ as $T_0 \ra \infty$ and an indirect use of Krylov-Bogliobov's theorem (such as a straightforward adaptation of Theorem 11.4.2 in \cite{Kallianpur}) asserts that $\mu^\zeta$ is the unique invariant measure of $(P^\zeta_t)_{t \geq 0}$, in the sense of (\ref{eq: invariant measure}). Due to the estimate (\ref{eq: uniform bound for tightness}) and the definition of $\mu^\zeta$ in (\ref{eq: invariant measure}), the simple application of monotone convergence shows, as in Lemma 3.4. in \cite{Cerrai2}, that there exists $C>0$ such that 
\begin{align} \label{eq: second moment invariant measure}
\int_{\RR^k} |y|^2 \mu^\zeta(dy)  \leq C (1 + |x|^2 + |y|^2).
\end{align} 
\no For any $x \in \RR^d$ we can define the averaged mixing coefficient
\begin{align} \label{eq: averaged coefficient}
\bar a(x) := \int_{\RR^k} a(x,y) \mu^\zeta (dy).
\end{align}
\no The proof of the following result concerning the Lipschitz continuity of $\bar a$ follows as in the arguments used to prove the stochastic averaging principle in \cite{Xu}.
\begin{proposition} \label{proposition: bar a is Lipschitz continuous}
Fix $T>0$ and $y \in \RR^k$. Let Hypothesis \ref{condition: hy coefficients existence uniqueness sol}-\ref{condition: dissipativity} hold for some $\nu \in \MM$. Then the function $\bar a$ defined by (\ref{eq: averaged coefficient}) is Lipschitz continuous. 
\end{proposition}
\no Proposition \ref{proposition: bar a is Lipschitz continuous} ensures that the averaged differential equation with initial condition $x \in \RR^d$,
\begin{align} \label{eq: the averaged ode}
\begin{cases}
\frac{d}{dt} \bar X^{0,x}_t &= \bar a(\bar X^{0,x}_t), \\
\bar X^{0,x}_0 &=x 
\end{cases}
\end{align}
has a unique solution $\bar X^{0,x} \in \CC([0,T]; \RR^d)$. \\

\no The following proposition  reads as a strong mixing property of the averaged coefficient $\bar a$ given by (\ref{eq: averaged coefficient}) and it plays a crucial role in the establishment of the large deviations principle for the family $(X^\varepsilon)_{\varepsilon>0}$. The proof follows is straghtforward and we refer the reader to \cite{Xu}.
\begin{proposition} \label{proposition: the averaging property for the averaged coefficient}
Fix $T>0$ and $y \in \RR^k$. Let Hypothesis \ref{condition: hy coefficients existence uniqueness sol}-\ref{condition: dissipativity} hold for some $\nu \in \MM$.Then there exists some function $\alpha:[0, \infty) \longrightarrow [0, \infty)$ such that $\alpha(T) \ra 0$ as $T \ra \infty$ and satisfying for any $t \in [0,T]$
\begin{align} \label{eq: mixing property of the averaged coefficient}
\bar \EE  \Big | \frac{1}{T} \int_t^{t+T} a(\zeta,Y^{\zeta,y}_s) ds - \bar a(x) \Big |^2 \leq \alpha(T) (1+ |x|^2 + |y|^2)
\end{align}
where the averaged coefficient $\bar a $ is defined by (\ref{eq: averaged coefficient}).
\end{proposition}

\section{The main theorems}
\subsubsection{The large deviations principle}
\no We make the following assumption on $\nu \in \MM$ that is used in the derivation of the large deviations principle for $(X^\varepsilon)_{\varepsilon>0}$.
\begin{condition} \label{condition: the measure}
The measure $\nu \in \MM$ is a L\'{e}vy measure on $(\RR^d \backslash \{0\}, \bB(\RR^d \backslash \{0\}))$, i.e. such that
$\int_{\RR^d \backslash \{0\}} (1 \wedge |z|^2) \nu(dz) < \infty$ that satisfies
\begin{align} \label{eq: integrability condition measure}
\int_{|z|\geq 1} e^{\alpha |z|^2} \nu(dz) < \infty, \quad \text{ for some } \alpha >0.
\end{align}
\end{condition}

\no In order to state the large deviations principle for the family of slow components $(X^\varepsilon)_{\varepsilon>0}$ given by (\ref{eq: the multiscale system})we fix some notation following mainly \cite{BDM11} and \cite{BCD13}. 

\no Fix $T>0$ and a measurable function $g: [0,T] \times \RR^d \backslash \{0\} \ra [0, \infty)$. We define the entropy functional  by 
\begin{align} \label{eq: entropy functional}
\eE_T(g) := \int_0^T \int_{\RR^d \backslash \{ 0\}} (g(s,z) \ln g(s,z) - g(s,z)+1) \nu(dz) ds.
\end{align}
For every $M \gqq 0$ we define the sublevel sets of the functional $\eE_T$ by 
\begin{align} \label{eq: the sublevel sets}
\mathbb{S}^M := \Big \{ g: [0,T] \times \RR^d \backslash \{0\} \lra [0, \infty) \mbox{ measurable }~\big|~ \eE_T(g) \lqq M \Big \} \quad \mbox{ and set } \quad \mathbb{S} := \bigcup_{M \gqq 0} \mathbb{S}^M. 
\end{align}
Given $T>0$, $x \in \RR^d$ and $g \in \mathbb{S}$ we 
consider the controlled integral equation
\begin{align} \label{eq: controlled ODE}
U^{g} (t;x) = x + \int_0^t \bar a(U^g(s;x))ds + \int_0^t \int_{\RR^d \backslash \{ 0\}} c(U^g (s;x),z)(g(s,z)-1)  \nu(dz)ds, ~t \in [0,T]. 
\end{align}
It is a standard fact
that the equation \eqref{eq: controlled ODE} has 
a unique solution $U^{g} \in C([0,T], \RR^d)$ and it satisfies the uniform bound 
\begin{align} \label{eq: uniform bound controlled odes}
\displaystyle \sup_{t \in [0,T]} \displaystyle \sup_{g \in \mathfrak{S}^M} |U^g (t;x)| < \infty \qquad \mbox{ for all }M>0. 
\end{align}
In particular, the map $\gG^{0, x}: \mathbb{S} \ra C([0,T], \RR^d)$, $g \mapsto \gG^{0,x}(g):= U^{g}(\cdot \,;x)$ is well-defined  for any fixed $x\in \RR^d$. \\

 \no For $\varphi \in C([0,T], \RR^d)$ we define
$\mathbb{S}_{\varphi, x} := \{ g \in \mathbb{S} ~|~ \varphi= \gG^{0,x}(g)\}$ and set   $\JJ_{x, T}: \DD([0,T], \RR^d) \ra [0, \infty]$ 
\begin{align} \label{eq: rate function}
\JJ_{x,T}(\varphi) := \displaystyle \inf_{g \in \mathfrak{S}_{\varphi, x}} \eE_T(g),
\end{align}
with the convention that $\inf \emptyset=\infty$. 

\begin{theorem} \label{thm: LDP full process} Let Hypotheses \ref{condition: hy coefficients existence uniqueness sol}-\ref{condition: the measure}
be satisfied for some $\nu\in \MM$, $T>0$, $x \in \RR^d$  and $y \in \RR^k$ fixed. Let 
$X^{\e} = (X^{\e}_t)_{t\in [0, T]}, \varepsilon>0$, 
be the slow component of the strong solution of (\ref{eq: the multiscale system}) given in Theorem \ref{theorem: existence uniqueness solution}. 
Then the family $(X^{\e})_{\e>0}$ satisfies a LDP with the good rate function $\JJ_{x,T}$ given by \eqref{eq: rate function} 
in the Skorokhod space $\DD([0,T], \RR^d)$. This means that, for any $a \gqq 0$ the 
sublevel set $\{ \JJ_{x,T} \lqq a\}$ is compact in $\DD([0,T],\RR^d)$ and for any 
$G \subset \DD([0,T], \RR^d)$ open and $F \subset \DD([0,T], \RR^d)$ closed,
\begin{align*}
\displaystyle \liminf_{\varepsilon \ra 0} \varepsilon \ln \bar \PP(X^{\varepsilon,x} \in G) &\gqq - \displaystyle \inf_{\varphi \in G} \JJ_{x,T}(\varphi) \text{ and } \\
\displaystyle \limsup_{\varepsilon \ra 0} \varepsilon \ln \bar \PP(X^{\varepsilon,x} \in F) & \lqq - \displaystyle \inf_{\varphi \in F} \JJ_{x,T}(\varphi).
\end{align*}
\end{theorem}
\no We prove Theorem \ref{thm: LDP full process} in Section \ref{section: proof LDP}. \\

\paragraph{Comment:} The reader can appreciate in the development of Section 4 the bypass of the usual exponential tightness between the family $(X^\varepsilon)_{\varepsilon >0}$ given by (\ref{eq: the multiscale system}) and $(\bar X^\varepsilon)_{\varepsilon >0}$ given by (\ref{eq: averaged processes}),  that would be of much harder verification,  with the key establishment of a controlled version of a stochastic averaging principle for the family $(\xX^\varepsilon)_{\varepsilon >0}$ defined in (\ref{eq: controlled averaged SDE}). The establishment of such stochastic averaging principle is the key result to prove Theorem \ref{thm: LDP full process} using the weak convergence approach.

\subsubsection{The first exit time problem in the small noise limit}
\no In this subsection we state the second main result of this work concerning the solution of the Kramers problem for $(X^\varepsilon)_{\varepsilon >0}$. 
\paragraph*{Further assumptions.}
\no We make the additional assumptions on the deterministic dynamical system (\ref{eq: the averaged ode}) and on the measure $\nu \in \MM$ as follows.
\begin{condition}\label{condition: domain}
\begin{enumerate}
\item  Let us consider a bounded domain $D \subset \RR^d$ 
  with $0 \in D$, $\partial D \in \cC^1$ and that $a$ is inward-pointing on $\partial D$, that is, 
\begin{align*}
\langle a(z), n(z) \rangle < 0, \quad \mbox{ for all } z \in \partial D, 
\end{align*}
\item The vector field $\bar a$ satisfies the following for some $c_1>0$ dissipativity condition
\begin{align} \label{eq: further dissipativity}
\langle \bar a(x) - \bar a(\tilde x),x - \tilde x \rangle \leq - c_1 |x- \tilde x|^2, \quad x, \tilde x \in D.
\end{align}
\end{enumerate}
\end{condition}
\begin{remark} \label{rem: on cond dynamics for first exit time}
\begin{enumerate}
\item Hypothesis \ref{condition: hy coefficients existence uniqueness sol} implies that $0 \in \RR^d$ is a critical point of (\ref{eq: the averaged ode}).
\item The assumption (\ref{eq: further dissipativity}) on Hypothesis \ref{condition: domain} implies that $D \bar a(x)$ is strictly negative definite for any $x \in D$. In the case (\ref{eq: the averaged ode}) is a gradient system given by $\bar a = \nabla U$ for some potential $U: \RR^d \longrightarrow [0,\infty)$ this is equivalent to uniform convexity. As a consequence of (\ref{eq: further dissipativity}) it follows that there exists $c_2>0$ such that $e^{- c_2 t} \bar X^{0,x}_t \ra 0$ as $ t \ra \infty$ for any $x \in D$.
 \item Hypothesis \ref{condition: domain} implies that 
the solution of (\ref{eq: the averaged ode}) is 
positive invariant on $\bar D$, that is, for all $x \in \bar D$, we have 
$X^{0, x}_t \in D$ for all $t \gqq 0$ 
and $X^{0, x}_t \rightarrow 0$ as $ t \ra \infty$.  
\end{enumerate}
\end{remark}

\no For every $\varepsilon>0$ the process $\varepsilon \tilde N^{\frac{1}{\varepsilon}}$ 
 is a compensated Poisson random measure defined on $(\bar \MM, \bB(\bar \MM), \bar \PP)$ with 
 compensator given by $ds \otimes \frac{1}{\varepsilon} \nu(dz)$ where $\nu \in \MM$ satisfies 
 the following assumption which replaces Hypothesis \ref{condition: the measure}. 
 \begin{condition} \label{condition:generalized statement- the measure nu} 
The measure $\nu \in \mathfrak{M}$ is non-atomic and satisfies the following conditions.
\begin{itemize}
\item[\textbf{E.1:}] The measure $\nu$ is a L\'{e}vy measure, i.e. $\int_{\RR^d \backslash \{0\}} (1 \wedge |z|^2) \nu(dz)< \infty$.
\item[\textbf{E.2:}] $\nu \in \MM$ satisfies  $\int_{B^c_1(0)} e^{\Gamma |z|^2} \nu(dz) < \infty$ for some $\Gamma>0$.
\item[\textbf{E.3:}] $\nu \in \MM$ is  symmetric.
\end{itemize}
\end{condition}

\paragraph*{Comment:} We refer the reader to \cite{Hogele AO} for a discussion on the assumptions made on the measure $\nu$ in Hypothesis \ref{condition:generalized statement- the measure nu}.
%\no We comment in Subsection \ref{subsection: comments kramers} the assumptions made above on $\nu \in MM$.

\no The following hypothesis is a continuity assumption on the controlled differential equation (\ref{eq: controlled ODE}) that is essential for solving the first exit time problem.
%and we comment it in Subsection \ref{subsection: comments kramers}.
\begin{condition} \label{condition: generalized statement-potential}  
For every $\rho_0>0$ there exist a constant 
$M>0$ and a non-decreasing function $\xi:[0, \rho_0] \ra \RR^{+}$ 
with $\lim_{\rho \ra 0} \xi(\rho)=0$ satisfying the following. 
For all $x_0, y_0 \in \RR^d$ such that $|x_0- y_0| \lqq \rho_0$ there exist 
$\Phi \in C([0, \xi(\rho_0)],\RR^d)$ and $g \in S^M$ 
such that $\Phi(\xi(\rho_0))=y_0$ and solving
\begin{align} \label{eq: generalized statement-cond on potential: controllability cond C1}
 \Phi(t) = x_0 + \int_0^t \bar a(\Phi(s)) ds + \int_0^t \int_{\RR^d \backslash \{ 0\}} c(\Phi(s),z) (g(s,z)-1) \nu(dz) ds, \quad t \in [0, \xi(\rho_0)].
\end{align}
\end{condition}

\begin{remark} Hypothesis \ref{condition: generalized statement-potential} is satisfied for a general class of L\'{e}vy measures $\nu \in \MM$ that satisfy Hypothesis \ref{condition:generalized statement- the measure nu} under the stricter assumptions that
\begin{itemize}
\item[\textbf{B.1:}] $\nu$ is a finite measure, $\nu(\RR^d \backslash \{0\}) < \infty$.
\item[\textbf{B.2:}] The measure $\nu$ is absolutely continuous with respect to the Lebesgue measure $dz$ 
on the measurable space $(\RR^d \backslash \{0\}, \bB(\RR^d \backslash \{0\}))$ and $\frac{d \nu}{dz}(z) \neq 0$ for 
every $z \in \RR^d \backslash \{0\}$.
\item[\textbf{B.3:}] We have that $\nu$ is symmetric.
\end{itemize}
 \end{remark}
We refer the reader to Proposition 15 in \cite{Hogele AO} where it is explained  how the set of assumptions in Hypothesis \ref{condition: generalized statement-potential} imply the solution of Kramers problem.

\no Given $\e>0$, $x \in D$ and $\nu \in \MM$ satisfying 
Hypotheses \ref{condition: hy coefficients existence uniqueness sol}, \ref{condition: dissipativity}, \ref{condition: domain} and \ref{condition:generalized statement- the measure nu}
we define the first exit time of the slow component $X^{\e,x}$ of (\ref{eq: the multiscale system})
from $D$ 
\begin{align} \label{eq: the first exit time}
\sigma^\e(x) := \inf \{ t \gqq 0 ~|~ X^{\e,x}_t \notin D \}  
\end{align}
and the first exit location $X^{\e, x}_{\sigma^\e(x)}$. 

\no The function $V$ quantifying the cost of shifting the intensity jump measure by a scalar control $g$ and steering $U^g(t;x)$ from its initial position $x$ to some $z\in \RR^d$ in cheapest time is defined as 
 \begin{align} \label{eq: running cost}
 V(x,z) := \displaystyle \inf_{T>0} \inf \Big \{ \JJ_{x,T} (\varphi) ~|~ \varphi \in \DD([0,T], \RR^d): 
 \varphi(T)=z \Big \} \quad \text{ for } x,z \in \RR^d.
 \end{align}
 The function $V(0,z)$ is called the quasi-potential of the stable state $0$ with potential height 
 \begin{align} \label{eq: the potential}
 \bar{V}:= \displaystyle \inf_{z \notin D} V(0,z).
 \end{align}
\no The following result solves the Kramers problem for $(X^\varepsilon)_{\varepsilon>0}$.
\begin{theorem} \label{thm: first exit time}
Let Hypotheses \ref{condition: hy coefficients existence uniqueness sol}, \ref{condition: dissipativity}, \ref{condition: domain}, \ref{condition:generalized statement- the measure nu}  and \ref{condition: generalized statement-potential} be satisfied.
Then $\bar V < \infty$ and we obtain the following result.
\begin{enumerate}
 \item \label{thm: exit time}
For any $x \in D$ and $\delta>0$, we have 
\begin{align} \label{eq: first exit time-limit1}
\displaystyle \lim_{\e \ra 0} \bar \PP \Big ( e^{\frac{\bar V - \delta}{\e}}< \sigma^\e(x) < e^{\frac{\bar V + \delta}{\e}} \Big )=1.
\end{align}
Furthermore, for all $x \in D$ it follows $\displaystyle \lim_{\e \ra 0} \e \ln \bar \EE[\sigma^\e(x)]= \bar V.$

\item \label{thm: location exit}
For any closed set $F \subset D^c$ satisfying $\displaystyle \inf_{z \in F} V(0,z) > \bar V$ 
and any $x \in D$, we have 
 \begin{align} \label{eq: first exit time-limit2}
 \displaystyle \lim_{\e \ra 0} \bar \PP \Big ( X^{\e,x}_{\sigma^\e(x)} \in F \Big )=0.
 \end{align}
In particular, if $\bar V$ is taken by a unique point $z^{*} \in D^c$, 
it follows, for any $x \in D$ and $\delta>0$, that 
\begin{align} \label{eq: first exit time-limit3}
\displaystyle \lim_{\e \ra 0} \bar \PP( |X^{\e,x}_{\sigma^\e(x)} - z^{*}| < \delta )=1.
\end{align}
\end{enumerate}
 \end{theorem}

\no The proof of ( \ref{thm: exit time}) in Theorem \ref{thm: exit time} follows as the proof of Theorem 3 in \cite{Hogele AO}. We present the proof of (\ref{thm: location exit}) from Theorem \ref{thm: exit time} in Subsection \ref{subsection: exit locus}.

\paragraph*{Comment:} We emphasize that the cost functional $V$  and the height potential $\bar V$ given by (\ref{eq: running cost}) and (\ref{eq: the potential}) respectively are written in terms of the good rate function $\JJ$ defined in (\ref{eq: rate function}) that is given by the minimization of the entropy functional $\eE$ over the controlled paths that solve (\ref{eq: controlled ODE}). This is clearly an example of complexity reduction in the solution of the Kramers problem for $(X^\varepsilon)_{\varepsilon>0}$ using the averaged stochastic dynamics of $(\bar X^{\varepsilon})_{\varepsilon>0}$.

\subsection{Examples}
\paragraph*{Strongly tempered exponentially light L\'{e}vy measures.} Hypothesis \ref{condition:generalized statement- the measure nu} covers a wide class of L\'{e}vy measures and we point out the following special benchmark cases. 
\begin{enumerate}
\item Our setting covers the simplest case of finite intensity super-exponentially light jump measures given by $\nu(dz)= e^{-\alpha |z|^2}$ for some $\alpha>1$. For every $\varepsilon>0$ the corresponding stochastic process $L^\varepsilon_t := \int_0^t \int_\XX z \tilde N^{\frac{1}{\varepsilon}}(ds,dz)$, $t \geq 0$ is a compensated compound Poisson process.
\item More generally Hypothesis \ref{condition:generalized statement- the measure nu} covers a class of L\'{e}vy measures that mimics the class of strongly tempered exponentially light measures introduced by Rosi\'{n}ski in \cite{Ros07}, however, with a Gaussian damping in order to satisfy (\ref{eq: integrability condition measure}). For the polar coordinate $r=|z|$ and any $A \in \bB(\XX)$ we define
\begin{align*}
\nu(A) = \int_{\RR^d \backslash \{0\}} \int_0^\infty \textbf{1}_{A}(rz) \frac{e^{- r^2}}{r^{\alpha'+1}}dr R(dz), \quad \alpha' \in (0,2),
\end{align*}
for some measure $R \in \MM$ such that $\int_{\RR^d \backslash \{0\}} |z|^{\alpha'} R(dz)< \infty$. We point out that, for every $\varepsilon>0$, the corresponding L\'{e}vy process $(L^\varepsilon_t)_{t \geq 0}$ differs from the compound Poisson process of the paragraph before not only from the fact that the corresponding jump measure has infinite total mass but also from the fact that although a compound Poisson process with positive jumps has almost surely nondecreasing paths, it does not have paths that are almost surely strictly increasing.
\end{enumerate}
 
 \paragraph*{Invariant measures for the Markov semigroup associated to the fast variable.}
 
 \no For every $\varepsilon>0$, $T>0$, $c>0$, $x \in \RR^d$ and $y \in \RR^k$ let us consider the multi-scale system where the fast variable $(Y^\varepsilon_t)_{t \in [0,T]}$ is decoupled from the slow component $(X^\varepsilon_t)_{t \in [0,T]}$ and is given by
 \begin{align*}
 \begin{cases}
 X^\varepsilon_t &=x + \int_0^t a(X^\varepsilon_s, Y^\varepsilon_s)ds + \varepsilon \int_0^t \int_{\RR^d \backslash \{0\}} c(X^\varepsilon_{s-},z) \tilde N^{\frac{1}{\varepsilon}}(ds,dz); \\
 Y^\varepsilon_t &=y - \frac{c}{\varepsilon} \int_0^t Y^\varepsilon_s ds + L^\varepsilon_t,
 \end{cases}
 \end{align*}
 under Hypotheses \ref{condition: hy coefficients existence uniqueness sol}-\ref{condition: dissipativity}. Here the stochastic process $(L^\varepsilon_t)_{ t \in [0,T]}$ is a pure jump process given by
 \begin{align*}
 L^\varepsilon_t := \int_0^t \int_{\RR^d \backslash \{0\}} z \tilde N^{\frac{1}{\varepsilon}}(ds,dz)
 \end{align*}
 where $\tilde N^{\frac{1}{\varepsilon}}$ is the compensated version of the Poisson random measure defined on $(\MM, \bB(\MM))$ with intensity given by $\frac{1}{\varepsilon} \nu \otimes ds$ for some $\nu \in \MM$ satisfying Hypothesis \ref{condition: the measure}. Therefore due to Theorem 17.5 in \cite{Sato}, freezing $\varepsilon=1$, there exists a unique invariant measure $\mu$ of $(Y^1_t)_{t \in [0,T]}$ given via its Fourier transform by
 \begin{align*}
 \hat \mu(dz) = \exp \Big ( \int_0^\infty \psi(e^{-cs}z) ds \Big )
 \end{align*}
 where $\psi$ (cf. Corollary 2.5. in \cite{Schilling}) is the L\'{e}vy symbol of $(L^1_t)_{t \in [0,T]}$ given by 
 \begin{align*}
 \EE \Big [ \exp(i \langle \xi L^1_t \rangle ) \Big ]= e^{- t \psi(\xi)} \quad \text{ for every } \xi \in \RR^d.
 \end{align*}

\section{Proof of the large deviations principle} \label{section: proof}
\subsection{The weak convergence in a nutshell}
\no The weak convergence approach to large deviations theory builds up in the equivalence between the definition of large deviations principle for a family $(X^\varepsilon)_{\varepsilon>0}$ defined on some common probability space $(\Omega, \fF, \PP)$ with values in a Polish (complete separable) space $E$ and the following definition.
\begin{definition} \label{definition: Laplace-Varadhan}
Let $(X^\varepsilon)_{\varepsilon>0}$ be a family of $E$-valued random variables. The family $(X^\varepsilon)_{\varepsilon>0}$ is said to satisfy the Laplace-Varadhan principle  in $E$ with the good rate function $I:E \longrightarrow [0,\infty]$ if for every $h \in C_b(E)$ the following holds:
\begin{align*}
\displaystyle \limsup_{\varepsilon \ra 0} \varepsilon \ln \EE \Big [ e^{- \frac{1}{\varepsilon} h(X^\varepsilon)} \Big ] & \leq - \displaystyle \inf_{x \in E} \{ h(x)+ I(x) \} \quad \text{and} \\
\displaystyle \liminf_{\varepsilon \ra 0} \varepsilon \ln \EE \Big [ e^{- \frac{1}{\varepsilon} h(X^\varepsilon)} \Big ] & \leq - \displaystyle \inf_{x \in E} \{ h(x)+ I(x) \}.
\end{align*}
\end{definition}
\no For a proof of this equivalence given by Varadhan we refer the reader to \cite{Dupuis Ellis}. 

\no Now,  writing for every $\varepsilon>0$ the shifted measure $\PP^\varepsilon:= \PP \circ (X^\varepsilon)^{-1}$ we have by Donsker-Varadhan's theorem \cite{Dupuis Ellis} 
 that for every $\varepsilon>0$ and every $h \in C_b(E)$ the following variational formula holds
\begin{align*}
- \varepsilon \ln \EE \Big [ e^{-\frac{1}{\varepsilon} h(X^\varepsilon)} \Big ] = \displaystyle \inf_{\QQ \in \pP(E)} \Big \{ h(x)  \QQ(dx) + R(\QQ || \PP^\varepsilon) \Big \}
\end{align*}
where $\pP(E)$ is the set of probability measures defined on $E$ and $R(\QQ || \PP^\varepsilon)$ is the relative entropy of the measure  $\QQ$ with respect to $\PP^\varepsilon$, i.e.
\begin{align*}
R(\QQ || \PP^\varepsilon) := \begin{cases} \int_E \ln \Big ( \frac{d \QQ}{d \PP^\varepsilon} \Big ) d \PP^\varepsilon \quad \text{if } \QQ \ll \PP^\varepsilon; \\
\infty \quad & \text{ otherwise}.
\end{cases}
\end{align*}

\no Due to the mentioned foundational results in order to show a large deviations principle to $(X^\varepsilon)_{\varepsilon>0}$ in the Polish space $E$ one has to show
\begin{align*}
\displaystyle \inf_{\QQ \in \pP(E)} \Big \{ h(x)  \QQ(dx) + R(\QQ || \PP^\varepsilon) \Big \} \ra \displaystyle \inf_{x \in E} \{h (x) +I(x) \}.
\end{align*}

%\no  In this spirit Budhiraja, Dupuis and Maroulas proved in \cite{BDM11} a variational formula for functionals of Poisson random measures that was successfully employed in the derivation of an abstract sufficient criteria for large deviations principles in a very generic context. The same criteria ables to reduce the proof of large deviations results for dynamical systems perturbed in small noise intensity by L\'{e}vy processes in finite dimensions (in the form of SDEs) and in infinite dimensions (in the form of SPDEs or more complex dynamical models). In order to state it we need to introduce some notation that we borrow mainly from \c{ref1, ref2} and from the recent monograph \cred{ref}. Introduce reference variational formula PRMs.

\no We consider the setup of Subsection \ref{subsection: prob setting}.
 Let $\bar \aA$ be the class of $\bar \pP \otimes \bB(\RR^d \backslash \{0\} \ \bB([0,\infty) ) )$ measurable scalar functions (random controls) $\varphi:[0,T] \times \RR^d \MM \longrightarrow [0,\infty)$ and let
\begin{align} \label{eq: Poisson energy}
L_T (\varphi) := \int_0^T \int_{\RR^d \backslash \{0\}} (\varphi \ln \varphi - \varphi +1) \nu(dz) ds.
\end{align}

\no We define the controlled random measure $N^\varphi$ with respect to $\bar N$ under $\bar \PP$:
\begin{align} \label{eq: controlled random measure}
N^\varphi ([0,T] \times A)(\omega) := \int_0^t \int_A \int_0^\infty \textbf{1}_{[0, \varphi(s,z)(\omega)]}(r) \bar N(ds,dz,dr), quad t \geq 0, A \in \bB(\RR^d \backslash \{0\}).
\end{align} 
\no  We refer the interested reader to \cite{BDM11}. In virtue of (\ref{eq: controlled random measure}) we note that, for every $\varepsilon>0$, the Poisson random measure $ N^{\frac{1}{\varepsilon}}$ can be interpreted as a controlled random measure by the scalar control function $\varphi(s,z)(\omega)=\frac{1}{\varepsilon}$. In this particular case the constant scalar control $\frac{1}{\varepsilon}$ is selecting in a non-anticipative way the jumps registered by the Poisson random measure that have intensity at most $\frac{1}{\varepsilon}$.
 
 \no We let $g:[0,T] \times \RR^d \backslash \{0\} \longrightarrow [0,\infty)$ be any measurable function (deterministic control) and $M>0$. We denote $S^M$ the sublevel $M>0$ for the energy (\ref{eq: Poisson energy}), i.e. we say $g \in S^M$ iif $L_T(g) \leq M$. Let $\SSS:= \bigcup_{M \geq 0} S^M$. We identify $g \in S^M$ with the measure $\nu^g \in \MM$ given by
 \begin{align} \label{eq: identification vague convergence}
 \nu^g(A) := \int_A g(s,z) \nu(dz)ds, \quad A \in \bB(\MM).
 \end{align}
Under the identification $S^M \simeq \{ \nu^g ~|~g \in S^M \}$ the space $S^M$ turns to be compact for the topology described by the vague convergence. We refer the reader to \cite{BCD13} for more details. 
 
 \no We write for any $M>0$ 
 \begin{align} \label{eq: random sublevel sets}
 \uU^M := \{ \varphi \in \bar \aA ~|~ \varphi \in S^M \quad \bar \PP-a.s. \}.
 \end{align}
 In \cite{BDM11} the authors prove the following variational formula, for every $F \in M_b(\MM)$ measurable bounded function defined on $\MM$,
 \begin{align*}
 - \ln \bar \EE [e^{F(N^1)}] = \displaystyle \inf_{\varphi \in \aA} \bar \EE [L_T(\varphi) + F(N^\varphi)].
 \end{align*}
 \no With the help of the variational formula given above  in \cite{BDM11},  the authors  introduced a sufficient abstract criteria for large deviations principles.
 \begin{condition} \label{condition: cond LDP}
  For any $\varepsilon>0$ we consider measurable maps $\gG^0: \SSS \longrightarrow \DD([0,T]; \RR^d)$ and $\gG^\varepsilon: \MM \longrightarrow \DD([0,T]; \RR^d)$ such that the following two conditions hold.
 \begin{enumerate}
 \item \textbf{Continuity statement of the limiting map.} Given $M \geq 0$ let $(g_n)_{n \in \NN} \subset S^M$ and $g \in S^M$ such that $\nu^g \rightharpoonup \nu^g$ as $n \ra \infty$ in the vague convergence. Then up to a subsequence it holds
 \begin{align*}
 \gG^0(g_n) \ra \gG^0(g) \quad \text{ as } n \ra \infty.
 \end{align*}
 \item \textbf{Weak law of large numbers for the random maps of shifted noises.} Given $M \geq 0$ let $(\varphi^\varepsilon)_{\varepsilon>0} \subset \uU^M$ and $\varphi \in \uU^M$ such that $\varphi^\varepsilon \Rightarrow\varphi$ as $\varepsilon \ra 0$. Then up to a subsequence the following holds:
 \begin{align*}
 \gG^\varepsilon  \Big ( \varepsilon N^{\frac{1}{\varepsilon} \varphi^\varepsilon} \Big ) \Rightarrow \gG^0(\varphi) \quad \text{ as } \varepsilon \ra 0.
\end{align*}  
 \end{enumerate}
 \end{condition}
 \no For every $\varepsilon>0$ let $Z^\varepsilon:= \gG^\varepsilon \Big( \varepsilon N^{\frac{1}{\varepsilon}} \Big ).$
 The next theorem states that Hypothesis \ref{condition: cond LDP} is sufficient to retrieve the large deviations principle for the family of random variables $(Z^\varepsilon)_{\varepsilon>0}$. We refer the reader to \cite{BDM11}.
 \begin{theorem} \label{theorem: thm LDP abstract}
 Under the Hypothesis \ref{condition: cond LDP} the family $(Z^\varepsilon)_{\varepsilon>0}$ satisfies a large deviations principle in the Skorokhod space $\DD([0,T]; \RR^d)$ with the good rate function
 \begin{align*}
\JJ: \DD([0,T]; \RR^d) & \longrightarrow [0,\infty]  \\
\JJ(\eta) & := \displaystyle \inf_{\gG^0(g)= \eta} L_T(g) \\
&= \inf_{\gG^0(g)= \eta} \int_0^T \int_{\RR^d \backslash \{0\}} (g(s,z) \ln g(s,z) - g(s,z) +1) \nu(dz) ds.
 \end{align*}
 \end{theorem}

\subsection{Verifying the abstract sufficient criteria.}
\no Let us fix $T>0$, $x \in \RR^d$, $y \in \RR^k$. We assume Hypotheses \ref{condition: hy coefficients existence uniqueness sol},  \ref{condition: dissipativity} and \ref{condition: the measure} to hold.  For every $\varepsilon>0$ let $(X^{\varepsilon,x,y}_t, Y^{\varepsilon,x,y}_t)_{t \in [0,T]}$ be the unique strong solution of (\ref{eq: the multiscale system}) in the sense of Definition \ref{definition: solution SDE}  given by Theorem \ref{theorem: existence uniqueness solution}.  We drop the dependence on $x$ and $y$. For any $\varepsilon>0$, Yamada-Watanabe's theorem ensures the existence of a measurable map $\gG^\varepsilon: \MM \longrightarrow \DD([0,T];\RR^d)$ such that 
\begin{align} \label{eq: Ito maps for the MDP}
X^{\varepsilon}:= \gG^\varepsilon( \varepsilon N^{\frac{1}{\varepsilon}}).
\end{align}
\no The proof of Theorem \ref{thm: LDP full process} consists in checking the conditions (1) and (2) of Hypothesis \ref{condition: cond LDP} for $(\gG^\varepsilon)_{\varepsilon>0}$ and $\gG^0: \mathbb{S}  \longrightarrow C([0,T];\RR^d)$, $\gG^0(g)= U^g$, with $U^g \in C([0,T];\RR^d)$ defined by the skeleton equation (\ref{eq: controlled ODE}). Hence Theorem \ref{theorem: thm LDP abstract} allows us to conclude. 

\subsubsection{The skeleton equations and the compactness condition}

 \no For any $x \in \RR^d$ and $g \in \mathbb{S}$ let us denote by $U^g=U^{g,x} \in \CC([0,T]; \RR^d)$ the unique solution of (\ref{eq: controlled ODE}). By definition we have  
\begin{align*}
\gG^0(g) = U^g.
\end{align*}
\begin{proposition} \label{prop: first condition LDP}
For every $M < \infty$ one has that the set
\begin{align*}
\KK_M := \Big \{ \gG^0  ( g  ) ~|~  g \in \mathbb{S}^M \Big \}
\end{align*}
is compact in $C([0,T]; \RR^d)$.
\end{proposition}
\no The proof is straightforward and we omit it. For details we refer the reader to Proposition 11 in \cite{Hogele AO}.
\begin{rem} \label{remark: the first condition MDP}
 Proposition \ref{prop: first condition LDP} can be rewritten as follows. Fix $0 \leq M < \infty$. 
 Let $(g_n)_{n \in \NN} \subset S^M$ such that $g_n \rightharpoonup g$ in the topology of the vague convergence as $n \ra \infty$ under the identification $\mathbb{S}^M \simeq \{ \nu^g ~|~g \in \mathbb{S}^M\}$ given above with $\nu^g$defined in (\ref{eq: identification vague convergence}) for any $g \in \mathbb{S}$. Therefore
\begin{align*}
\gG^0(g_n) \ra \gG^0(g) \quad \text{ as } n \ra \infty \quad \text{ in the uniform topology}.
\end{align*}
This form of restating Proposition \ref{prop: first condition LDP} corresponds to the verification of the  first condition in Hypothesis \ref{condition: cond LDP} for $\gG^0$.
\end{rem}

\subsubsection{Strategy of the proof of the second condition in Hypothesis \ref{condition: cond LDP}.}

\no In order to apply Theorem \ref{theorem: thm LDP abstract} and conclude the large deviations principle given in Theorem \ref{thm: LDP full process} we proceed verifying the second condition in Hypothesis \ref{condition: cond LDP} for  $\gG^0$ and the family $\{ \gG^\varepsilon: \MM \longrightarrow \DD([0,T]; \RR^d)\}_{\varepsilon>0}$. For every $M>0$ recall the random sublevel set $\uU^M$  given by (\ref{eq: random sublevel sets}) and for every $\varepsilon>0$ let $\varphi^\varepsilon \in \uU^M$. Set $\tilde \varphi^\varepsilon= \frac{1}{\varphi^\varepsilon}$. The definition of $\tilde \varphi^\varepsilon$ makes sense since one has that $\bar \PP$-a.s.  $\varphi^\varepsilon \in \aA_b$ holds.  For any $t \in [0,T]$ we define the $\FF$-martingale 
\begin{align*}
\eE(\tilde \varphi^\varepsilon)(t) &:= \exp \Big ( \int_0^t \int_{\RR^d \backslash \{0\}} \int_0^{\frac{1}{\varepsilon}} \ln \tilde \varphi^\varepsilon(s,z) \bar N(ds,dz,dr) \Big ) \\
& + \int_0^t \int_{\RR^d \backslash \{0\}} \int_0^{\frac{1}{\varepsilon}} (- \tilde \varphi^\varepsilon(s,z)+1) ds \nu(dz) dr \Big ).
\end{align*}
\no  Girsanov's theorem stated in the form of Theorem III.3.24 of \cite{Jacod Shiryaev} ensures that $(\eE(\varphi^\varepsilon)(t))_{t \in [0,T]}$ is an $\FF$-martingale. Hence the probability measures defined on $(\bar \MM, \bB(\bar \MM))$ by
\begin{align*}
\QQ^\varepsilon_T(G) := \int_G  \eE(\varphi^\varepsilon)(T) d \bar \PP, \quad \text{ for all } G \in \bB(\bar \MM)
\end{align*}
are absolutely continuous with respect to $\bar \PP$. Under $\QQ^\varepsilon_T$ the stochastic process 
 $\varepsilon N^{\frac{1}{\varepsilon} \varphi^\varepsilon}$ is a random measure with the same law of $\varepsilon N^{\frac{1}{\varepsilon}}$ under $\bar \PP$. We recall that
\begin{align*}
N^{\frac{1}{\varepsilon} \varphi^\varepsilon}((0,t] \times U) := \int_0^t \int_U \int_0^\infty \textbf{1}_{[0, \frac{1}{\varepsilon} \varphi^\varepsilon]}(r) \bar N(ds,dz,dr), \quad t \geq 0, U \in \bB(\RR^d \backslash \{0\}).
\end{align*} 
\no  For any $(x,y) \in \RR^d \times \RR^k$, $M>0$, $\varepsilon>0$ and $\varphi^\varepsilon \in \uU^M$ we define the slow controlled process $(\xX^\varepsilon(t))_{t \in [0,T]}$ and the fast controlled process $(\yY^\varepsilon(t))_{t \in [0,T]}$ given as the strong solutions with respect to $\bar \PP$ (since $\QQ^\varepsilon_T \ll \bar \PP$) of the following stochastic differential system of equations
\begin{align} \label{eq: controlled SDEs multiscale}
\begin{cases}
\xX^\varepsilon_t &= x + \displaystyle \int_0^t \Big ( a(\xX^\varepsilon_s, \yY^\varepsilon_s) + \int_{\RR^d \backslash \{0\}} c(\xX^\varepsilon_s,z)(\varphi^\varepsilon(s,z)-1) \nu(dz) \Big )ds + \varepsilon \int_0^t c(\xX^\varepsilon_{s-},z) \tilde N^{\frac{\varphi^\varepsilon}{\varepsilon}}(ds,dz) \\
\yY^\varepsilon_t &= y +  \displaystyle \frac{1}{\varepsilon} \int_0^t \Big ( f(\xX^\varepsilon_s, \yY^\varepsilon_s) + \int_{\RR^d \backslash \{0\}} h(\xX^\varepsilon_s, \yY^\varepsilon_s,z) (\varphi^\varepsilon(s,z)-1) \nu(dz) \Big )ds + \int_0^t h(\xX^\varepsilon_{s-}, \yY^\varepsilon_{s-},z) \tilde N^{\frac{\varphi^\varepsilon}{\varepsilon}}(ds,dz).
\end{cases}
\end{align}

\no For every $T>0$, $x \in \RR^d$, $M>0$, $\varepsilon>0$ and $\varphi^\varepsilon \in \uU^M$ we define $(\bar \xX^\varepsilon(t))_{t \in [0,T]}$ the fast averaged controlled process as the strong solution under $\bar \PP$ of the controlled stochastic differential equation 
\begin{align} \label{eq: controlled averaged SDE}
\bar \xX^\varepsilon_t &=x + \int_0^t \Big ( \bar  a (\xX^\varepsilon_s) + \int_{\RR^d \backslash \{0\}} c(\xX^\varepsilon_s,z)(\varphi^\varepsilon(s,z)-1)\Big )ds + \varepsilon \int_0^t c(\bar \xX^\varepsilon_{s-},z) \tilde N^{\frac{\varphi^\varepsilon}{\varepsilon}}(ds,dz)
\end{align}.

\no For every $\varepsilon>0$ and $M>0$ let $(\varphi^\varepsilon)_{\varepsilon>0} \subset \uU^M$ such that $\varphi^\varepsilon \Rightarrow \varphi$ as $\varepsilon \ra 0$ in the vague convergence topology in $S^M$. The conclusion in the second statement in Hypothesis \ref{condition: cond LDP} for $(\gG^\varepsilon)_{\varepsilon>0}$ and $\gG^0$ reads as $\xX^\varepsilon \Rightarrow \bar X$, as $\varepsilon \ra 0$ in law where $\bar X \in  \CC([0,T]; \RR^d)$ solves uniquely
\begin{align} \label{eq: the controlled averaged eq for LDP}
\bar X^\varphi(t)&= \int_0^t \bar a(\bar X(s))ds + \int_0^t \int_{\RR^d \backslash \{0\}} c(\bar X(s),z)(\varphi(s,z)-1) \nu(dz)ds 
\end{align}
In order to prove that $\xX^\varepsilon \Rightarrow \bar X$, as $\varepsilon \ra 0$ we proceed as follows.
\begin{itemize}
\item[1.] This step passes through two intermediary tasks. Firstly one shows that the laws of $(\bar \xX^\varepsilon)_{\varepsilon>0}$ are tight in $\pP(C([0,T]; \RR^d))$ (since compact sets in the topology generated by the uniform convergence are also compact sets in the Skorokhod topology). Then it follows that there exists $\tilde \xX \in C([0,T];\RR^d)$ such that $\bar \xX^\varepsilon \Rightarrow \tilde \xX$ as $\varepsilon \ra 0$. Passing to the pointwise limit in the equation (\ref{eq: controlled averaged SDE}) satisfied by $\bar \xX^\varepsilon$ and due to the uniqueness of solution of (\ref{eq: the controlled averaged eq for LDP}) we conclude that $\tilde \xX= \bar X^\varphi$.
\item[2.] We prove the following strong (controlled) averaging principle:
\begin{align} \label{eq: controlled averaging principle}
\displaystyle \lim_{\varepsilon \ra 0} \bar \PP \Big ( \displaystyle \sup_{t \in [0,T]} |\xX^\varepsilon(t) - \bar \xX^\varepsilon(t)| > \delta\Big )=0, \quad \text{ for any } \delta>0.
\end{align}
From the limit above and Theorem 4.1. in \cite{Billingsley}, commonly known as Slutzsky's theorem, we can identify $\bar \xX$ as the weak limit of $(\xX^\varepsilon)_{\varepsilon}$ as $\varepsilon \ra 0$.
\end{itemize}

\no The first point in the strategy announced above is proved in the following proposition. The proof follows analogous as the proof of Proposition 12 in \cite{Hogele AO}.
\begin{proposition} \label{prop: second condition LDP} Given $M>0$ and $\varepsilon>0$ let $\varphi^\varepsilon \in \uU^M$ such that $\varphi^\varepsilon \Rightarrow \varphi$ as $\varepsilon \ra 0$ in the vague convergence in $S^M$. Therefore $\gG^0(\varphi)= \bar X^\varphi$ is a limit point in law of $\gG^\varepsilon(\varepsilon N^{\frac{\varphi^\varepsilon}{\varepsilon}})$ in $\DD([0,T];\RR^d)$.
\end{proposition}

\subsubsection{The controlled averaging principle}
\no In order to conclude one has resumed to prove (\ref{eq: controlled averaging principle}). For that purpose we use a localization technique based in the following useful result. The proof uses a Bernstein-type inequality for c\`{a}dl\'{a}g martingales given in \cite{DZ01} and the proof follows analogously as the proof of Proposition 9 in \cite{Hogele AO}.
\begin{proposition} \label{prop: localization}
Let the hypotheses of Theorem \ref{thm: LDP full process} be satisfied. For every $M>0$, $(\varphi^\varepsilon)_{\varepsilon>0} \subset \uU^M$, any function $\rR:(0,1] \longrightarrow (0,\infty)$ such that $\rR(\varepsilon) \ra \infty$ as $\varepsilon \rR^2(\varepsilon) \ra 0$ as $\varepsilon \ra 0$, $x \in \RR^d$ and $T>0$ we have that there exists some $\varepsilon_0 \in (0,1)$ and $C>0$ such that for every $\varepsilon < \varepsilon_0$ the following holds
\begin{align} \label{eq: localization}
\bar \PP \Big ( \displaystyle \sup_{t \in [0,T]} |\xX^\varepsilon_t| \vee |\bar \xX^\varepsilon_t| > \rR(\varepsilon)\Big ) \leq 2 e^{- \frac{1}{2} \rR(\varepsilon)} + C \varepsilon \rR(\varepsilon).
\end{align}
\end{proposition} 
\no For every $\varepsilon>0$ and any $\rR:(0,1] \longrightarrow (0,\infty)$ satisfying the limits $\rR(\varepsilon) \ra \infty$ and $\varepsilon \rR^2(\varepsilon) \ra 0$ as $\varepsilon \ra 0$ such as in the statement of Proposition \ref{prop: localization} let us define the $\FF$-stopping time
\begin{align} \label{eq: stopping time}
\tilde \tau^\varepsilon_{\rR(\varepsilon)} := \displaystyle \inf \{ t \geq 0 ~|~ \xX^\varepsilon_t  \notin B_{\rR(\varepsilon)}(0)\} \wedge  \displaystyle \inf \{ t \geq 0 ~|~ \bar \xX^\varepsilon_t  \notin B_{\rR(\varepsilon)}(0) \} \wedge T.
\end{align}

\no The following a-priori bounds are used in the sequel for the proof of the controlled averaging principle (\ref{eq: controlled averaging principle}). The proof is straightforward and follows from several applications of the Ito's formula and the Burkholder-Dqvis-Gundy's inequalities.

\begin{proposition} \label{prop: a priori bounds controlled processes}
Let $M>0$. Fix a function $\rR:(0,1] \longrightarrow [0,\infty)$ satisfying the assumptions of Proposition \ref{prop: localization} and for every $\varepsilon>0$ let $\tilde \tau^\varepsilon_{\rR(\varepsilon)}$ be defined in (\ref{eq: stopping time}). Under the assumptions of Theorem \ref{thm: LDP full process} we have that there exists some $\varepsilon_0>0$ such that for every $\varepsilon< \varepsilon_0$ implying
\begin{align} \label{eq: Gamma1}
\Gamma_1(M) := \displaystyle \sup_{0 < \varepsilon < \varepsilon_0} \displaystyle \sup_{ \varphi \in \uU^M} \Big ( \bar \EE \Big [ \displaystyle \sup_{0 \leq t \leq \tilde \tau_{\rR(\varepsilon)} |\xX^\varepsilon_t|^2} \Big ] + \displaystyle \sup_{0 \leq t \leq T} \bar \EE \Big [ |\yY^\varepsilon_t|^2 \textbf{1}_{\{ \tilde \tau^\varepsilon_{\rR(\varepsilon)}\}} \Big ] \Big ) < \infty
\end{align}
and
\begin{align} \label{eq: Gamma2}
\Gamma_2(M):= \displaystyle \sup_{0 < \varepsilon < \varepsilon_0} \displaystyle \sup_{\varphi \in \uU^M} \bar \EE \Big [ \displaystyle \sup_{0 \leq t leq \tilde \tau^\varepsilon_{\rR(\varepsilon)}} |\bar \xX^\varepsilon_t|^2 \Big ] < \infty.
\end{align}
\end{proposition}

\no We follow the technique introduced in \cite{Khasminskii} with the the required modifications to our setting in order to deal with the nonlocal components of the auxiliary processes $(\xX^\varepsilon)_{\varepsilon>0}$ and $(\yY^\varepsilon)_{\varepsilon>0}$ given by (\ref{eq: controlled SDEs multiscale}).  \\

\no Let $[0,T]$ be divided into intervals of the same length parametrized for every $\varepsilon>0$
\begin{align} \label{eq: parametrization- the Delta} 
\Delta= \Delta(\varepsilon):= \varepsilon^\gamma  |\ln \varepsilon|^p, \quad \text{for some } \gamma \in \Big (0, \frac{1}{2} \Big ) \quad \text{and } p>0.
\end{align}
 We note the following convergences
\begin{align} \label{eq: parametrizations- the conv Delta}
&\Delta(\varepsilon) \ra 0; \quad \text{ and }  \quad \frac{\Delta(\varepsilon)}{\varepsilon} \ra \infty \quad \text{ as } \varepsilon \ra 0.
\end{align}
For any $t \in [0,T]$ we denote $t_\Delta := \left \lfloor{\frac{t}{\Delta}} \right \rfloor \Delta$.\\

\no We construct the auxiliary processes $(\hat \yY^\varepsilon(t))_{t \in [0,T]}$ and $(\hat \xX^\varepsilon(t))_{t \in [0,T]}$ by means of the following equations: for any $t \in [0,T]$ let 
\begin{align} \label{eq: Khasminkii -the fast variable with frozen component}
\hat \yY^\varepsilon_t&= \yY^\varepsilon_{t_\Delta} + \frac{1}{\varepsilon} \int_{t_\Delta}^t \Big ( f(\xX^\varepsilon_{t_\Delta}, \hat \yY^\varepsilon_s)+ \int_{\RR^d \backslash \{0\}} h( \xX^\varepsilon_{t_\Delta}, \hat \yY^\varepsilon_s,z) (\varphi^\varepsilon(s,z)-1) \nu(dz) \Big ) ds \nonumber \\
&  + \int_{t_\Delta}^t \int_{\RR^d \backslash \{0\}} h(\xX^\varepsilon_{t_\Delta-}, \hat \yY^\varepsilon_{s-},z) \tilde N^{\frac{1}{\varepsilon} \varphi^\varepsilon}(ds,dz)
\end{align}
and
\begin{align} \label{eq: Khasminkki- the slow variable with frozen component}
\hat \xX^\varepsilon_t & = x + \int_0^t  \Big (a(\xX^\varepsilon_{s_\Delta}, \yY^\varepsilon_s) + \int_{\RR^d \backslash \{0\}} c(\xX^\varepsilon_{s},z) (\varphi^\varepsilon(s,z)-1) \nu(dz) \Big ) ds \nonumber \\
& + \varepsilon \int_0^t \int_{\RR^d \backslash \{0\}} c(\xX^\varepsilon_{s-},z) \tilde N^{\frac{1}{\varepsilon} \varphi^\varepsilon}(ds,dz).
\end{align}

\paragraph*{Comment:} The modifications from the continous Gaussian regime to the pure jump noise regime in (\ref{eq: Khasminkii -the fast variable with frozen component}) and (\ref{eq: Khasminkki- the slow variable with frozen component}) are natural with the respective frozen variables done in the same way. 

\no We enunciate the following list of results that are used in the sequel to prove (\ref{eq: controlled averaging principle}). The proofs are modifications from the arguments used in \cite{Khasminskii} to the Poissonian case.

\begin{lemma} \label{lemma: Khasminkii segment process estimate}
For every $\varepsilon>0$ let $\rR(\varepsilon)>0$ and $\Delta(\varepsilon)>0$ fixed as above. Then for every $\delta>0$ the following 
\begin{align} \label{eq: Khasmikkii segment process estimate}
\bar \PP \Big ( \displaystyle \sup_{0 \leq t \leq \tilde \tau^\varepsilon_{\rR(\varepsilon)}} |\xX^\varepsilon_{t} - \xX^\varepsilon_{t_\Delta}| > \delta \Big ) & \lesssim_\varepsilon  \Xi(\varepsilon) \ra 0, \quad \text{as } \varepsilon \ra 0,
\end{align}
\end{lemma}

\begin{lemma} \label{lemma: Khasminkki}
For every $\varepsilon>0$ let $\rR(\varepsilon)$ fixed as in Proposition \ref{prop: localization}  and $\Delta(\varepsilon)$ given by (\ref{eq: parametrization- the Delta}). Then the following convergence holds: 
\begin{align} \label{eq: lemma Khasmkinkii- fast variable convergence}
\displaystyle \sup_{0 \leq t \leq T} \bar \EE \Big  [|\yY^\varepsilon(t) - \hat \yY^\varepsilon(t)| \textbf{1}_{\{T< \tilde \tau^\varepsilon_{\rR(\varepsilon)} \}} \Big ] \lesssim_\varepsilon C(\varepsilon) \ra 0 \quad \text{as } \varepsilon \ra 0.
\end{align}
for some $C(\varepsilon) \ra 0$ as $\varepsilon \ra 0$ uniformly in the initial condition $(x, y) \in \RR^d \times \RR^k$.
\end{lemma}
\no The previous a-priori bounds (\ref{eq: Khasmikkii segment process estimate}) and (\ref{eq: lemma Khasmkinkii- fast variable convergence}) imply the following.

\begin{proposition} \label{prop: Khasminkii averaging controlled-part1}
For any $\delta>0$ we have
\begin{align} \label{eq: Khasminkii limit1}
\displaystyle \limsup_{\varepsilon \ra 0} \bar \PP \Big ( \displaystyle \sup_{0 \leq t \leq \tilde \tau^\varepsilon_{\rR(\varepsilon)}} |\xX^\varepsilon_t - \hat \xX^\varepsilon_t| > \frac{\delta}{2}\Big )=0.
\end{align}
\end{proposition}

\begin{proof}
\no The definitions of $(\yY^\varepsilon(t))_{t \in [0,T]}$ and  $(\hat \yY^\varepsilon(t))_{t \in [0,T]}$ given in (\ref{eq: controlled SDEs multiscale})and (\ref{eq: Khasminkii -the fast variable with frozen component}) respectively combined with Hypothesis \ref{condition: dissipativity}  yield for every $\varepsilon>0$ and $t \in [0,T]$ that
\begin{align*}
\hat \xX^\varepsilon_t - \xX^\varepsilon_t =\int_0^t \Big ( a(\xX^\varepsilon_{s_\Delta}, \hat \yY^\varepsilon_s) - a(\xX^\varepsilon_s, \yY^\varepsilon_s)  \Big )ds \\
 \leq L \int_0^t |\xX^\varepsilon_{s_\Delta}- \xX^\varepsilon_s| ds + L \int_0^t |\hat \yY^\varepsilon_s - \yY^\varepsilon_s| ds .
\end{align*}
\no The asymptotic behaviour (\ref{eq: parametrizations- the conv Delta}) of $\Delta(\varepsilon)>0$ fixed in (\ref{eq: parametrization- the Delta}) combined with Lemma \ref{lemma: Khasminkii segment process estimate}  of Lemma \ref{lemma: Khasminkki} yield some $C=C(L,T)>0$ such that 
\begin{align*}
\bar \PP \Big ( \displaystyle \sup_{0 \leq t \leq \tilde \tau^\varepsilon_{\rR(\varepsilon)}} |\hat \xX^\varepsilon_t - \xX^\varepsilon_t| > \frac{\delta}{2L} \Big ) & \leq \bar \PP \Big ( \int_0^{T \wedge \tilde \tau^\varepsilon_{\rR(\varepsilon)}} |a(\xX^\varepsilon_{s_\Delta}, \hat \yY^\varepsilon_s) - a(\xX^\varepsilon_s, \yY^\varepsilon_s)| ds > \frac{\delta}{2 L} \Big ) \\
%& \leq  \bar \PP \Big ( \displaystyle \sup_{0 \leq t \leq T \wedge \tilde \tau^\varepsilon_{\rR(\varepsilon)}} | \xX^\varepsilon_{t_\Delta} - \xX^\varepsilon_t| > \frac{\delta}{2} L \Big ) \\
&+ \bar \PP \Big ( \int_0^T |\hat \yY^\varepsilon_s - \yY^\varepsilon_s|^2  \textbf{1}_{\{ T< \tilde \tau^\varepsilon(\rR(\varepsilon))\}} ds > \frac{\delta}{2L} \Big ) \\
& \lesssim_\varepsilon \Xi(\varepsilon) + \frac{2}{\delta} \int_0^T \bar \EE \Big [ |\hat \yY^\varepsilon_s- \yY^\varepsilon_s|^2 \textbf{1}_{\{ T< \tilde \tau^\varepsilon_{\rR(\varepsilon)}\}} ds\Big ] \\
&\lesssim_\varepsilon \Xi(\varepsilon) + C(\varepsilon)  \ra 0 \text{ as } \varepsilon \ra 0.
\end{align*}
This finishes the proof of (\ref{eq: Khasminkii limit1}).
\end{proof}

\begin{proposition} \label{prop: Khasminkii averaging controlled-part2}
For any $\delta>0$ we have
\begin{align} \label{eq: Khasminkii limit 2}
\displaystyle \limsup_{\varepsilon \ra 0} \bar \PP \Big ( \displaystyle \sup_{0 \leq t \leq \tilde \tau^\varepsilon_{\rR(\varepsilon)}} |\hat \xX^\varepsilon_t - \bar \xX^\varepsilon_t|> \frac{\delta}{2} \Big )=0.
\end{align}
\end{proposition}
\begin{proof}
 \no For every $\varepsilon>0$, $t \in [0,T]$, $\xi \in \RR^d$ and  $\varphi^\varepsilon \in \uU^M$, we define the function
\begin{align*}
b^\varepsilon(\xi)(t):= \int_0^t   \int_{\RR^d \backslash \{0\}} c(\xi,z) (\varphi^\varepsilon(s,z)-1) \nu(dz) ds.
\end{align*}
\no The definitions of $(\xX^\varepsilon(t))_{t \in [0,T]}$ and $(\hat \xX^\varepsilon(t))_{t \in [0,T]}$ given in (\ref{eq: controlled SDEs multiscale}) and respectively in  (\ref{eq: Khasminkki- the slow variable with frozen component}) combined with the definition of $b^\varepsilon$ given above imply for every $t \in [0,T]$ and $\varepsilon>0$   the following identity $\bar \PP$-a.s. on the event $\{ T < \tilde \tau^\varepsilon_{\rR(\varepsilon)} \}$:
\begin{align} \label{eq: Khasminkii-final1}
\hat \xX^\varepsilon_t - \bar \xX^\varepsilon_t&= \int_0^t \Big ( b^\varepsilon(\hat \xX^\varepsilon_s) - b^\varepsilon(\bar \xX^\varepsilon_s)\Big ) ds \nonumber \\
& + \int_0^t \Big ( a(\xX^\varepsilon_{s_\Delta}, \hat \yY^\varepsilon_s) - \bar a(\xX^\varepsilon_s) \Big ) ds \nonumber \\
& + \int_0^t \Big ( \bar a(\xX^\varepsilon_s) - \bar a(\hat \xX^\varepsilon_s) \Big ) ds + \int_0^t \Big ( \bar a(\hat \xX^\varepsilon_s) - \bar a(\bar \xX^\varepsilon_s) \Big ) ds \nonumber \\
& + \varepsilon \int_0^t \int_\XX \Big (c(\xX^\varepsilon_{s-},z) - c(\bar \xX^\varepsilon_{s-},z) \Big ) \tilde N^{\frac{1}{\varepsilon} \varphi^\varepsilon}(ds,dz).
\end{align}
\no Hypothesis \ref{condition: dissipativity},  Proposition \ref{proposition: bar a is Lipschitz continuous} and (\ref{eq: Khasminkii-final1}) yield some constant $C=C(L,T)>0$ such that on the event $\{ T < \tilde \tau^\varepsilon_{\rR(\varepsilon)}\}$ we have $\bar \PP$-a.s.
\begin{align*}
\displaystyle \sup_{0 \leq s \leq t} |\hat \xX^\varepsilon_s - \bar \xX^\varepsilon_s|^2 & \leq C \Big ( \int_0^t \displaystyle \sup_{0 \leq u \leq s} |\hat \xX^\varepsilon_u - \bar \xX^\varepsilon_u|^2 ds+ \displaystyle \sup_{0 \leq s \leq t} \Big | \int_0^s \Big ( a(\xX^\varepsilon_{u_\Delta}, \yY^\varepsilon_u) - \bar a(\xX^\varepsilon_u) \Big )du  \Big |^2  \\
&+ \displaystyle \sup_{ t \in [0,T]} |J^\varepsilon(t)|^2    \textbf{1}_{\{ T < \tilde \tau^\varepsilon_{\rR(\varepsilon)}\}} \Big ),
\end{align*}
where for any $\varepsilon>0$ we write
\begin{align*}
J^\varepsilon(t) &:=   \varepsilon  \int_0^t \int_{\RR^d \backslash \{0\}} \Big (c(\xX^\varepsilon_{s-},z) - c(\bar \xX^\varepsilon_{s-},z) \Big ) \tilde N^{\frac{1}{\varepsilon} \varphi^\varepsilon}(ds,dz).
\end{align*}
\no Gronwall's lemma implies for any $\varepsilon>0$ that
\begin{align} \label{eq: Khasminkii-final2}
&\displaystyle \sup_{-\tau \leq t \leq T} |\xX^\varepsilon_t - \xX^\varepsilon_t|^2 \textbf{1}_{\{ T < \tilde \tau^\varepsilon_{\rR(\varepsilon)}\}} \nonumber\\
& \leq e^{CT} \Big ( \displaystyle \sup_{0 \leq s \leq t} \Big | \int_0^s \Big ( a(\xX^\varepsilon_{u_\Delta}, \yY^\varepsilon_u) - \bar a(\xX^\varepsilon_u) \Big )du  \Big |^2 \textbf{1}_{\{ T < \tilde \tau^\varepsilon_{\rR(\varepsilon)}\}}
 + \displaystyle \sup_{0 \leq t \leq T \wedge \tilde \tau^\varepsilon_{\rR(\varepsilon)}} |J^\varepsilon(t)|^2 \Big ).
\end{align}
\no The estimate (\ref{eq: Khasminkii-final2}) yields for any $\delta>0$ 
\begin{align} \label{eq: Khasminkii-final3}
\bar \PP \Big ( \displaystyle \sup_{0 \leq t \leq \tilde \tau^\varepsilon_{\rR(\varepsilon)}} |\hat \xX^\varepsilon_t - \bar \xX^\varepsilon_t| > \frac{ \delta}{2} \Big ) & \leq \bar \PP \Big ( \displaystyle \sup_{0 \leq s \leq t} \Big | \int_0^s \Big ( a(\xX^\varepsilon_{u_\Delta}, \yY^\varepsilon_u) - \bar a(\xX^\varepsilon_u) \Big )du  \Big |^2 \textbf{1}_{\{ T < \tilde \tau^\varepsilon_{\rR(\varepsilon)}\}}  > \frac{\delta^2  e^{-  CT}}{8} \Big ) \nonumber \\
& + \bar \PP \Big ( \displaystyle \sup_{0 \leq t \leq T \wedge \tilde \tau^\varepsilon_{\rR(\varepsilon)}} |J^\varepsilon(t)|^2  > \frac{\delta^2  e^{-  CT}}{8} \Big ).
\end{align}
\no   Burkholder-Davis-Gundy's inequalities and Lemma 8 in \cite{Hogele AO} imply that there exists some constant $C_1=C_1(\delta, C,, \Gamma_1, \Gamma_2,M)>0$, where $\Gamma_1, \Gamma_2>0$ are given by (\ref{eq: Gamma1}) and (\ref{eq: Gamma2}) of Proposition \ref{prop: a priori bounds controlled processes}, that  may change from line to line, such that
\begin{align}  \label{eq: Khasminkii-final3B}
 \bar \PP \Big ( \displaystyle \sup_{0 \leq t \leq T \wedge \tilde \tau^\varepsilon_{\rR(\varepsilon)}} |J^\varepsilon(t)|^2  > \frac{\delta^2   e^{-  CT}}{8} \Big ) &  \leq \frac{8  e^{ CT}}{\delta^2} \bar \EE \Big [ \displaystyle \sup_{0 \leq t \leq T \wedge \tilde \tau^\varepsilon_{\rR(\varepsilon)}} |J^\varepsilon(t)|^2 \Big ] \nonumber \\
 & \leq \frac{\varepsilon}{C_1}  \displaystyle \sup_{g \in \sS^M_{+,\varepsilon}} \int_0^T \int_{\RR^d \backslash \{0\}} |z|^2 g(s,z) \nu(dz) ds \nonumber \\
 & \leq C \varepsilon  \ra 0.
\end{align}
\no  We estimate now the first term in the right hand-side of (\ref{eq: Khasminkii-final3}). For every $\varepsilon>0$ and $t \in [0,T]$ we write $\bar \PP$-a.s. on the event $\{ T < \tilde \tau^\varepsilon_{\rR(\varepsilon)}\}$ 
 \begin{align} \label{eq: Khasminkii-final4}
 \int_0^t \Big ( a(\xX^\varepsilon_{s_\Delta}, \hat \yY^\varepsilon_s) - \bar a(\xX^\varepsilon_s)) ds \Big ) &= \sum_{k=0}^{\left \lfloor{\frac{t}{\Delta}} \right \rfloor-1} \int_{k \Delta}^{(k+1)\Delta} \Big (a(\xX^\varepsilon_{k \Delta}, \yY^\varepsilon_s) - \bar a(\xX^\varepsilon_{k \Delta}) \Big ) ds \nonumber \\
 & + \sum_{k=0}^{\left \lfloor{\frac{t}{\Delta}} \right \rfloor -1} \int_{k \Delta}^{(k+1) \Delta} \Big ( \bar a(\xX^\varepsilon_{k \Delta}) - \bar a(\xX^\varepsilon_s) \Big ) ds \nonumber \\
 &+ \int_{t_\Delta}^t \Big ( a(\xX^\varepsilon_{s_\Delta}, \hat \yY^\varepsilon_s) - \bar a(\xX^\varepsilon_s) \Big ) ds \nonumber \\
 & := I^\varepsilon_1 + I^\varepsilon_2 + I^\varepsilon_3.
 \end{align}
 It follows from (\ref{eq: Khasminkii-final4}) that
 \begin{align} \label{eq: Khasminkii-final5}
  \bar \PP \Big ( \displaystyle \sup_{0 \leq s \leq t} \Big | \int_0^s \Big ( a(\xX^\varepsilon_{u_\Delta}, \yY^\varepsilon_u) - \bar a(\xX^\varepsilon_u) \Big )du  \Big |^2 \textbf{1}_{\{ T < \tilde \tau^\varepsilon_{\rR(\varepsilon)}\}}  > \frac{\delta^2   e^{-  CT}}{8} \Big ) & \leq  \bar \PP \Big ( \displaystyle \sup_{0 \leq t \leq T} |I^\varepsilon_1(t)| \textbf{1}_{\{T < \tilde \tau^\varepsilon_{\rR(\varepsilon)}\}} > \frac{\delta e^{- CT}}{2 \sqrt{2} }\Big )  \nonumber \\
 & +  \bar \PP \Big ( \displaystyle \sup_{0 \leq t \leq T} |I^\varepsilon_2(t)| \textbf{1}_{\{T < \tilde \tau^\varepsilon_{\rR(\varepsilon)}\}} > \frac{\delta  e^{-  CT}}{2 \sqrt{2}   }\Big ) \nonumber \\
 &  + \bar \PP \Big ( \displaystyle \sup_{0 \leq t \leq T} |I^\varepsilon_3(t)| \textbf{1}_{\{T < \tilde \tau^\varepsilon_{\rR(\varepsilon)}\}} > \frac{\delta a(\varepsilon) e^{-  CT}}{2 \sqrt{2}}\Big ) .
\end{align}
\paragraph*{Estimating $I^\varepsilon_2$.} We observe that for any $\varepsilon>0$ 
\begin{align*}
I^\varepsilon_2 = \int_0^{t_\Delta} \Big ( \bar a(\xX^\varepsilon_{s_\Delta}) - \bar a(\xX^\varepsilon_s)\Big ) ds.
\end{align*}
Proposition \ref{proposition: bar a is Lipschitz continuous} and Lemma \ref{lemma: Khasminkii segment process estimate} implies for some $C_2=C(T)>0$, any $\delta>0$ and $\varepsilon>0$ small enough that
\begin{align} \label{eq: Khasminkii-final I2}
 \bar \PP \Big ( \displaystyle \sup_{t \in [0,T]} |I^\varepsilon_1(t)| \textbf{1}_{\{T < \tilde \tau^\varepsilon_{\rR(\varepsilon)}\}} > \frac{\delta a(\varepsilon) e^{-  CT}}{2 \sqrt{2} }\Big )  & \leq \bar \PP \Big ( \displaystyle \sup_{0 \leq t \leq \tilde \tau^\varepsilon_{\rR(\varepsilon)}} \int_0^{t_\Delta} |\xX^\varepsilon_{s_\Delta}- \xX^\varepsilon_s| > C_4 a(\varepsilon)\Big ) \lesssim_\varepsilon \Xi (\varepsilon) \ra 0 \quad \text{as } \varepsilon \ra 0.
\end{align}

\paragraph*{Estimating $I^\varepsilon_3$.} Hypothesis \ref{condition: hy coefficients existence uniqueness sol}, Proposition \ref{proposition: bar a is Lipschitz continuous} and Proposition \ref{prop: a priori bounds controlled processes} yield some constant $C_3=C_3(L,\Gamma_1(M))>0$ that may change from line to line such that, for every $\varepsilon>0$ small enough and any $\delta>0$, one has
\begin{align} \label{eq: Khasminkii-final I3}
\bar \PP \Big ( \displaystyle \sup_{t \in [0,T]} |I^\varepsilon_3(t)| \textbf{1}_{\{ T < \tilde \tau^\varepsilon_{\rR(\varepsilon)}\}} > \frac{\delta a(\varepsilon)e^{- CT}}{2 \sqrt{2}} \Big ) 
& \leq C_3 \bar \EE \Big [ \displaystyle \sup_{0 \leq t \leq \tilde \tau^\varepsilon_{\rR(\varepsilon)}} \Big | \int_{t_\Delta}^t \Big ( a(\xX^\varepsilon_{s_\Delta}, \hat \yY^\varepsilon_s) - \bar a(\xX^\varepsilon_s) \Big ) ds \Big |^2 \Big ] \nonumber \\
& \leq C_3 \bar \EE \Big [ \int_0^T \Big ( 1+ |\xX^\varepsilon_s|^2 + |\xX^\varepsilon_{s_\Delta}|^2_\infty + |\yY^\varepsilon_s|^2 \Big ) \textbf{1}_{\{ T < \tilde \tau^\varepsilon_{\rR(\varepsilon)}\}} ds \Big ]  \nonumber \\
& \lesssim_\varepsilon \Delta(\varepsilon) \ra 0, \quad \text{as } \varepsilon \ra 0,
\end{align}
due to (\ref{eq: parametrizations- the conv Delta}).

\paragraph*{Estimating $I^\varepsilon_1$.} 
\no We construct a new process $Z:= \yY^{\varepsilon} (\xX^\varepsilon_{k \Delta}, \yY^{\varepsilon}(k \Delta))$ where the notation that is displayed here stresses out that the process is the fast variable process $\yY^\varepsilon$ with frozen slow component $\xX^\varepsilon_{k \Delta}$ and initial condition $\yY^{\varepsilon}(k \Delta))$. It is a classical  fact in the course of the Khasminkii's technique employed in \cite{Khasminskii} for the proof of the strong averaging principle that 
for every $s \in [0, \Delta]$ we have
\begin{align*}
(\xX^\varepsilon_{k \Delta}, \yY^\varepsilon_{s+k \Delta})=^d \Big ( \xX^\varepsilon_{k \Delta}, \yY^\varepsilon  ( \xX^\varepsilon_{k \Delta}, \yY^\varepsilon(k \Delta)  ) \Big ( \frac{s}{\varepsilon}\Big )\Big ).
\end{align*}
We may assume in addition that the fabricated noises above are independent of $\xX^\varepsilon_{k \Delta}$ and $\yY^\varepsilon(k \Delta)$. For the proof of the statements above we refer the reader to Section 5 in \cite{Xu}.

\no Hence Proposition \ref{proposition: the averaging property for the averaged coefficient} together with the Markov property of $(X^\varepsilon_t, Y^\varepsilon(t))_{t \in [0,T]}$ implies for every $k=0,\dots, \left \lfloor{\frac{t}{\Delta}} \right \rfloor$ the following:
\begin{align} \label{eq: Khasminkii-final8}
\bar \EE \Big [ \Big | \int_{k \Delta}^{(k+1) \Delta} \Big (a(\xX^\varepsilon_{k \Delta}, \hat \yY^\varepsilon(s)) - \bar a(\xX^\varepsilon_{k \Delta}) \Big ) ds \Big | \Big ] 
& \leq  \Delta \bar \EE \Big [ \frac{\varepsilon}{\Delta} \Big | \int_0^{\frac{\Delta}{\varepsilon}} \Big ( a(\xX^\varepsilon_{k \Delta}, Z(s)) - \bar a(\xX^\varepsilon_{k \Delta}) \Big ) ds \Big | \Big ] \nonumber \\
& = \Delta \bar \EE \Big [ \bar \EE \Big [ \Big | \frac{\varepsilon}{\Delta} \int_0^{\frac{\Delta}{\varepsilon}} a(\zeta, Z^{\zeta,y}) - \bar a(\zeta) \Big | \Big | (\zeta,y)=(\xX^\varepsilon_{k \Delta}, \yY^\varepsilon(k \Delta))\Big ]\Big ] \nonumber \\
& \leq \Delta \alpha \Big ( \frac{\Delta}{\varepsilon}\Big ) \Big ( 1 + \bar \EE ||\xX^\varepsilon_{k \Delta}|| + \bar \EE [|\yY^\varepsilon(k \Delta)|]\Big ).
\end{align}
\no Proposition \ref{proposition: the averaging property for the averaged coefficient}, Proposition \ref{prop: a priori bounds controlled processes}, (\ref{eq: parametrizations- the conv Delta}) and (\ref{eq: Khasminkii-final8}) yield, for any $\delta>0$ and $\varepsilon>0$ sufficiently small, that
\begin{align} \label{eq: Khasminkii-final9}
\bar \PP \Big ( \displaystyle \sup_{0 \leq t \leq T} |I^\varepsilon_1| \textbf{1}_{\{ T < \tilde \tau^\varepsilon_{\rR(\varepsilon)}\}} > \frac{\delta a(\varepsilon) e^{- CT}}{2 \sqrt{2}}\Big ) 
& \lesssim_\varepsilon  \bar \EE \Big [ \displaystyle \sup_{0 \leq t \leq T} |I^\varepsilon_1(t)|^2 \textbf{1}_{\{ T < \tilde \tau^\varepsilon_{\rR(\varepsilon)}\}}\Big ] \nonumber \\
& \lesssim_\varepsilon \sum_{k=0}^{ \left \lfloor{\frac{T}{\Delta(\varepsilon)}} \right \rfloor} \Big ( \bar \EE \Big | \int_{k \Delta}^{(k+1)\Delta} (a(\xX^\varepsilon_{k \Delta}, \hat \yY^\varepsilon_s) - \bar a(\xX^\varepsilon_{k \Delta})) \textbf{1}_{\{ T < \tilde \tau^\varepsilon_{\rR(\varepsilon)}\}} ds \Big | \Big )^2  \nonumber\\
& \lesssim_\varepsilon \Delta(\varepsilon) \alpha \Big ( \frac{\Delta(\varepsilon)}{\varepsilon}\Big ) \ra 0 \text{ as } \varepsilon \ra 0.
\end{align}
The convergence above follows from the choice of the parametrization $\Delta= \Delta(\varepsilon)$ fixed in (\ref{eq: parametrization- the Delta}) and $\alpha$ constructed in Proposition \ref{proposition: the averaging property for the averaged coefficient}
\end{proof}

\begin{theorem} \label{thm: controlled averaging principle}
Let the hypotheses of Theorem \ref{thm: LDP full process} to hold. Then we have
\begin{align*}
\displaystyle \limsup_{\varepsilon \ra 0} \bar \PP \Big ( \displaystyle \sup_{0 \leq t \leq \tilde \tau^\varepsilon_{\rR(\varepsilon)}} |\xX^\varepsilon_t - \bar \xX^\varepsilon_t| > \delta \Big )=0.
\end{align*}
\end{theorem}
\begin{proof}
\no For any $\varepsilon>0$ fix $\rR(\varepsilon)>0$ such as in Proposition \ref{prop: localization} and recall the definition of $\tilde \tau^\varepsilon_{\rR(\varepsilon)}$ in (\ref{eq: stopping time}). \\

\no For any $\delta>0$ we have
\begin{align} \label{eq: Khasminkii- the initial estimate}
\displaystyle \limsup_{\varepsilon \ra 0} \bar \PP \Big ( \displaystyle \sup_{0 \leq t \leq T} |\xX^\varepsilon(t) - \bar \xX^\varepsilon(t)| > \delta \Big ) &\leq  \displaystyle \limsup_{\varepsilon \ra 0} \bar \PP \Big ( \displaystyle \sup_{0 \leq t \leq \tilde \tau^\varepsilon_{\rR(\varepsilon)}} |\xX^\varepsilon_t - \hat \xX^\varepsilon_t| > \frac{\delta}{2}  \Big ) \nonumber \\
& + \displaystyle \limsup_{\varepsilon \ra 0} \bar \PP \Big ( \displaystyle \sup_{0 \leq \leq \tilde \tau^\varepsilon_{\rR(\varepsilon)}} |\hat \xX^\varepsilon_t - \bar \xX^\varepsilon_t| > \frac{\delta}{2}  \Big ) \nonumber \\
&+ \displaystyle \limsup_{\varepsilon \ra 0} \bar \PP \Big ( \tilde \tau^\varepsilon_{\rR(\varepsilon)} \leq T \Big ) \nonumber \\
& =0,
\end{align}
due to Proposition \ref{prop: localization}, Proposition \ref{prop: Khasminkii averaging controlled-part1} and Proposition \ref{prop: Khasminkii averaging controlled-part2}.
\end{proof}

\section{Proof of Theorem \ref{thm: LDP full process}} \label{section: proof LDP}
We recall the collection of measurable maps $(\gG^\varepsilon)_{\varepsilon>0}$ introduced in (\ref{eq: Ito maps for the MDP}) and $\gG^0$ defined by means of the skeleton equation (\ref{eq: controlled ODE}). We note that Proposition \ref{prop: first condition LDP} reads as the Condition 1 of Hypothesis \ref{condition: cond LDP} for $(\gG^\varepsilon)_{\varepsilon>0}$ and $\gG^0$. Proposition \ref{prop: second condition LDP} combined with Theorem \ref{thm: controlled averaging principle} yield, due to Slutzky's theorem, that Condition 2 of  Hypothesis \ref{condition: cond LDP} is verified for $(\gG^\varepsilon)_{\varepsilon>0}$ and $\gG^0$. Hence, the result follows from Theorem \ref{theorem: thm LDP abstract}.

\section{The Kramers problem- the exit locus} \label{section: Kramers}

\subsection{Auxiliary results}

\no In the sequel we use the continuity of the LDP of $(X^{\e, x})_{\e>0}$ 
with respect  to the initial condition $x \in \RR^d$ that follows directly from Theorem 4.4 in \cite{Maroulas}.

\begin{proposition} \label{chpt2: proposition uniform large dev principle} 
 Given $T>0$ and $x \in D$ let $F \subset \DD([0,T], \mathbb{R}^d)$ 
 be closed and $G \subset \DD([0,T], \mathbb{R}^d)$ open with respect to the 
 Skorokhod topology. Then we have
\begin{align}
&\limsup_{\substack{\e\rightarrow 0\\ y \ra x}}  \e \ln  \bar \PP (X^{ \e, y} \in F) \lqq - \inf_{f \in F} \JJ_{x, T}(f), \label{eq: uniform ldp a}\\
&\liminf_{\substack{\e\rightarrow 0\\ y \ra x}}  \e \ln \bar  \PP (X^{ \e, y} \in G) \gqq - \inf_{g \in G} \JJ_{x, T}(g). \label{eq: uniform ldp b}
\end{align}
\end{proposition}

\no  Proposition \ref{chpt2: proposition uniform large dev principle}  implies
the uniform LDP for 
$(X^{\e,x})_{\e>0}$ when the initial state $x \in K$ for $K \subset D$ a 
closed (and bounded) set. The proof is virtually the 
same as the one given in the Brownian case and we omit it. 
We refer the reader to Corollary 5.6.15 in \cite{DZ98}.

\begin{corollary}\label{chpt2: corol Ldp uniform in compact sets of initial states} 
Let $T>0$, $K \subset D$ be compact, $F \subset \DD ([0,T], \mathbb{R}^d)$ closed, 
$G \subset \DD ([0,T], \mathbb{R}^d)$ open  with respect to  the $J_1$ topology and $x \in D$. 
Then it follows
\begin{align*}
& \limsup_{\e\rightarrow 0} \sup_{y \in K} \e \ln \bar  \PP (X^{\e, y} \in F) \lqq - \inf_{y \in K, f \in F} \JJ_{y, T}(f), \\
& \liminf_{\e\rightarrow 0} \inf_{y \in K} \e \ln \bar  \PP (X^{ \e, y} \in G) \gqq - \inf_{y \in K, g \in G} \JJ_{y, T}(g). 
\end{align*}
\end{corollary} 
\no In the sequel this result is applied to the first exit time problem of $X^{\e, x}$ from $D$. 

\begin{lemma} \label{chpt2: lemma first lemma in lower bound first exit time} 
 For any $x \in D$ and $\rho > 0$ such that $ \bar{B}_\rho (0) \subset D$ we have
\begin{align*}
\lim_{\e \rightarrow 0} \bar \PP (X_{\vt^\e_{\rho}(x) }^{\e, x} \in \bar{B}_{\rho} (0)) = 1.
\end{align*}
\end{lemma}
\no For a proof of this lemma we refer the reader to Lemma 20 in \cite{Hogele AO}.

\
\no For a given $\rho>0$ such that $\bar B_\rho(0) \subset D$ we define
\begin{align} \label{eq: stopping time for localization}
\vartheta^\varepsilon_\rho(x) := \inf \{ t \geq 0 ~|~ |X^{\varepsilon,x}_t| \leq \rho  \text{ or } X^{\varepsilon,x}_t \in D^c \}.
\end{align}

\begin{lemma} \label{chpt2: lemma second lemma lower bound first exit time} 
For any $\rho > 0$ and $c>0$ 
there exists  $\xi(\rho)>0$ such that  $t\in [0, \xi(\rho)]$ implies
\begin{align*}
\limsup_{\e \rightarrow 0} \e \ln \sup_{x \in D} \bar \PP ( \displaystyle \sup_{t\in [0,\xi(\rho)]} 
|X_t^{ \e,x} -x| \gqq \rho) < -c.
\end{align*}
\end{lemma}
\no For a proof of the next result we refer the reader Lemma 22 in \cite{Hogele AO}
\begin{lemma} \label{chpt2: lemma third lemma lower bound first exit time} Let $F \subset D^c$ closed. Then
\begin{align*}
\lim_{\rho \rightarrow 0} \limsup_{\e \rightarrow 0} \e \ln \sup_{x \in B_{\rho}(0)} \bar \PP (X_{\vt^x_\rho}^{ \e, x} \in F)  
\lqq - \inf_{z \in F} V(0,z).
\end{align*}
\end{lemma} 
\no For a proof of the previous lemma we refer the reader Lemma 23 in \cite{Hogele AO}.

\subsection{Proof of (\ref{thm: location exit}) in Theorem \ref{thm: exit time}} \label{subsection: exit locus}
 \begin{proof}
 \begin{enumerate}
\item Let $F \subset D^c$  be a given closed set such that
 \begin{align*}
 V_F:= \displaystyle \inf_{z \in F} V(0,z) > \bar V.
 \end{align*}
 The case $\displaystyle \inf_{z \in F} V(0,z)= \infty$ is coverede by some $V_{F} \in (\bar V, \infty)$ instead. 
 
 \no Let $x \in D$. For some $\rho>0$ let $\vartheta^\varepsilon_\rho(x)$ be the $\bar \FF$-stopping time given by (\ref{eq: stopping time for localization}). 
 
\no By definition of $\sigma^\varepsilon$ and $\vartheta^\varepsilon_\rho(x)$ one has the elementary estimate
\begin{align*}
\bar \PP \Big ( X^\varepsilon_{\sigma^\varepsilon(x)} \in F \Big ) & \leq  \bar \PP \Big ( X^\varepsilon_{\tau^\varepsilon_{\rho}(x)} \notin B_\rho (0) \Big ) + \displaystyle \sup_{y \in B\rho (0)} \bar \PP \Big ( X^\varepsilon_{\tau^\varepsilon_\rho (y)} \in F \Big ).
\end{align*} 
 \no The first term of the r.h.s. of the last estimate converges to 0 as $\varepsilon \ra 0$ by Lemma \ref{chpt2: lemma first lemma in lower bound first exit time}. It remains to prove that
 \begin{align*}
 \displaystyle \lim_{\varepsilon \ra 0} \displaystyle \sup_{y \in B_\rho (0) } \bar  \PP \Big ( X^\varepsilon_{\tau^\varepsilon_\rho (y)} \in F \Big ) =0.
 \end{align*}
 \no Fix now $\eta \in (0, \frac{V_F - \bar V}{3})$ and according to Lemma \ref{chpt2: lemma third lemma lower bound first exit time} we choose $\rho, \varepsilon_{0}>0$ such that
 \begin{align} \label{eq: locus exit first estimate}
 \displaystyle \sup_{|z| \leq 2 \rho} \bar \PP \Big ( X^\varepsilon_{\tau^\varepsilon_\rho(z)} \in F \Big ) \leq e^{- \frac{V_F-\eta}{\varepsilon}}, \quad \text{ for every } \varepsilon< \varepsilon_0.
 \end{align}

\no  We now consider the following auxiliary stopping times constructed as follows. Due to Hypothesis \ref{condition: domain} there is $\rho'>0$ such taht $\bar B_{\rho'}(0) \subset D$ and $\displaystyle \sup_{x \in B_{\rho'}(0)} \langle b(x), n(x) <0$ and let $\rho>0$ such that $B_\rho(0) \subset B_{\rho'}(0)$. For every $x \in D$ we define recursively
 \begin{align} \label{eq: markov chain for the first exit times} 
\zeta^x_0 &:= 0 \non  \quad \text{and for any } k \in \NN\\
\vt_{k, \rho}^x &:= \inf \{ t \gqq \zeta^x_k ~|~ X^{ \e, x}_t \in \bar B_{\rho}(0) \cup D^c \}, \non \\
\zeta^x_{k+1} & := 
\begin{cases} \infty, \quad &\text{if } X^{ \e, x}_{\vt_{k, \rho}^x} \in D^c, \\
\inf \{  t \gqq \vt_{k, \rho}^x ~|~ X^{ \e, x}_t \in \bar B^c_{\rho'} (0) \}, \quad &\text{if } X^{ \e, x}_{\vt_{k, \rho}^x} \in \bar B_\rho(0).\\
\end{cases} 
\end{align}

\paragraph*{Comment:} This modified Markov chain approximation from the Gaussian case takes into account the topological particularities of the Skorokhod space on which we have the LDP.  In addition, the effect of the $\frac{1}{\varepsilon}$ acceleration
of the jump intensity enters as follows. The asymptotically exponentially negligible error estimates concerning the stickyness of the diffusion to its initial value,  which in the classical Brownian case are valid for time intervals of order 1,  in our case only hold for time intervals of order $\varepsilon$.  We account for a tiltting between entering the balls of radius $\rho$ and $\rho'$ centered on the stable state of the deterministic dynamical system and the excursions outside the domain $D$.

\no By construction $(\zeta^x_k)_{k \in \NN}$ and $(\vt_{k, \rho}^x)_{k \in \NN}$ we have $\bar \PP$-a.s. for all $k \in \NN$ 
\begin{align*}
  \zeta^x_{k} \lqq \vartheta^x_{k, \rho} \lqq \zeta^x_{k+1} \lqq \vartheta^x_{k+1, \rho}.
\end{align*}
Since $\rho'> \rho$ we have that $\zeta^x_{k+1} > \vartheta^x_{k}$ if $X^{\varepsilon,x}_{\vartheta^x_{k}} \in \bar B_\rho(0)$. Hence $(\vt^x_{k, \rho})_{k \in \NN}$ is an increasing sequence of $(\fF_t)-$stopping times. Since
the process $(X^{\e,x}_t)_{t \gqq 0}$ has the strong Markov property with respect to $(\fF_t)_{t \gqq 0}$ it follows that 
$(X^{ \e, x}_{\vt_{k, \rho}^x})_{k \in \NN}$ is a Markov chain and  $\sigma^\varepsilon(x) = \vartheta^x_{\ell,\rho}$ for some (random) $\ell \in \NN$ with the convention 
$X^{\e, x}_{\vt_{\ell, \rho}^x} := X^{ \e, x}_{\sigma^{\e} (x)}$ if $\vt_{\ell, \rho}^x = \infty$. 
 
 \no From the strong Markov property of $(X^\varepsilon_t)_{T \in [0,T]}$ it follows for all $T>0$ and $k \in \NN$ that
 \begin{align*}
 \displaystyle \sup_{z \in D} \bar \PP (\vartheta^z_{\rho,k} \leq k T) \leq k \displaystyle \sup_{z \in D} \bar  \PP \Big ( \displaystyle \sup_{0 \leq t \leq T} |X^{\varepsilon,z}_t - z| \geq \rho \Big ).
 \end{align*}
 
 \no In what follows we apply Lemma \ref{chpt2: lemma second lemma lower bound first exit time} with the choice of $\rho>0$ done as before and $c:= V_F - \eta$. One can choose $T:= T(\rho, V_f, \eta) < \infty$ such that
 \begin{align*}
 \displaystyle \sup_{z \in D} \bar \PP \Big (\displaystyle \sup_{0 \leq t \leq T} |X^{\varepsilon,z}_t - z| \geq \rho \Big ) \leq e^{- \frac{V_F - \bar V}{\varepsilon}}, \quad \varepsilon \leq \varepsilon_0. 
 \end{align*}
\no Consequently there exists $T'< \infty$ such that for every $\varepsilon> \varepsilon_0$ with $\varepsilon_0>0$ small enough one has 
\begin{align} \label{eq: exit locus second estimate}
\displaystyle \sup_{z \in D} \PP \Big ( \vartheta^z_{,\rho, k} \leq K T'  \Big ) \leq k e^{- \frac{V_F- \eta}{\varepsilon}}. 
\end{align}
\no Due to the strong Markov property one observes that
\begin{align*}
\Big \{ \sigma^\varepsilon (y) > \vartheta^y_{m-1, \rho} \Big \} \in \fF_{\vartheta^y_{m-1, \rho}} \subset \fF_{\zeta^y_m}
\end{align*}
and $\vartheta^y_{m, \rho} = \zeta^y_m + \vartheta^\varepsilon_{\rho} (X^{\varepsilon,y}_{\zeta^y_m}).$
Therefore for any $y \in D$ it follows that
\begin{align*}
\bar \PP \Big ( X^{\varepsilon,y}_{\vartheta^x_{m, \rho}} \in F \text{ and } \sigma^\varepsilon (y) > \vartheta^x_{m-1, \rho} \Big )& = \int_{\{ \sigma^\varepsilon(y)> \vartheta^x_{m-1, \rho} \} } \bar  \EE^{\fF_{\vartheta^y_{m, \rho}}} \textbf{1}_{\{ X^{\varepsilon,y}_{\vartheta^y_{m, \rho}} \in F\}} d \bar \PP\\
& = \int_{\{ \sigma^\varepsilon(y)> \vartheta^x_{m-1, \rho} \} } \bar \PP \Big ( X^{\varepsilon,y}_{\vartheta^\varepsilon_\rho (X^{\varepsilon,y}_{\vartheta^y_{m, \rho}})} \Big ) d \bar \PP \\
& \leq \bar  \PP \Big ( \sigma^\varepsilon(y) > \vartheta^y_{m-1, \rho} \Big ) \displaystyle \sup_{|z| \leq 2 \rho} \bar \PP \Big ( X^\varepsilon_{\vartheta^\varepsilon_\rho (z)} \in F \Big ).
\end{align*}
 \no For all $y \in B_\rho (0)$,  $k \in \NN$ and $\varepsilon> \varepsilon_0$,  where $\varepsilon_0>0$ is fixed as earlier,  due to the fact $\sigma^\varepsilon(y) > \vartheta^y_0=0$,  the definitions of $(\vt^x_{k, \rho})_{k \in \NN}$,  $\vartheta^\varepsilon_\rho(x)$, $\sigma^\varepsilon$ and the strong Markov property of $X^{\varepsilon,x}$ yield 
\begin{align} \label{eq: exit locus third estimate}
\bar \PP \Big ( X^\varepsilon_{\sigma^\varepsilon(y)} \in F \Big ) & \leq  \bar \PP \Big (              X^\varepsilon_{\sigma^\varepsilon(y)} \in F,   \quad \sigma^{\varepsilon}(y) < \vartheta^y_{k, \rho} \Big ) + \bar \PP \Big (  \sigma^\varepsilon(y) > \vartheta^y_{k, \rho}  \Big ) \nonumber \\
 & \leq \sum_{m=1}^k \bar \PP \Big ( X^\varepsilon_{\vartheta^y_{m, \rho}} \in F, \quad \sigma^\varepsilon(y) > \vartheta^y_{m-1, \rho} \Big ) + \bar \PP \Big ( \sigma^\varepsilon(y) > \vartheta^y_{k, \rho} \Big ) \nonumber \\
& \leq  k \displaystyle \sup_{|z| \leq 2 \rho} \bar \PP \Big ( X^\varepsilon_{\vartheta^\varepsilon_\rho (z)} \in F \Big ) + \bar \PP \Big ( \sigma^\varepsilon(y) > k T' \Big ) + \bar  \PP \Big ( \vartheta^y_{k, \rho} \leq kT' \Big ) \nonumber \\
& \leq 2 k e^{- \frac{V_f - \eta}{\varepsilon}} + \bar  \PP \Big ( \vartheta^y_{k, \rho} > k T' \Big ).
\end{align}

\no In the last estimate above we combined  (\ref{eq: locus exit first estimate}) with (\ref{eq: exit locus second estimate}).

\no The first statement in Theorem \ref{thm: exit time} implies that for every $\delta>0$ and $\varepsilon < \varepsilon_0$ with $\varepsilon_0>0$ small enough we have
\begin{align} \label{eq: exit locus 4th estiamte}
\displaystyle \sup_{x \in D} \bar \EE \Big [\sigma^\varepsilon(x) \Big ] \leq  e^{\frac{\bar V+ \frac{\eta}{2}}{\varepsilon}}.
\end{align}
 \no We parametrize $k=k_\varepsilon = \left \lceil{e^{\frac{\bar V +2 \eta}{\varepsilon}}}\right \rceil $. By the previous choice of  $\eta>0$ one has $\bar{V} - V_F + 3 \eta<0$. Hence (\ref{eq: exit locus third estimate}) and (\ref{eq: exit locus 4th estiamte}) yield
\begin{align*}
\displaystyle \limsup_{\varepsilon \ra 0} \displaystyle \sup_{y \in B_\rho(0)}  \bar \PP \Big ( X^\varepsilon_{\sigma^\varepsilon(y)} \in F \Big ) 
& \leq \displaystyle \limsup_{\varepsilon \ra 0} \Big ( 2 k(\varepsilon) e^{- \frac{V_F- \eta}{\varepsilon}} + \frac{1}{k(\varepsilon)T'} e^{\frac{\bar V + \eta}{\varepsilon}} \Big ) \\
& \leq \displaystyle \limsup_{\varepsilon \ra 0} \Big ( 2 e^{\frac{\bar V - V_F + 3 \eta}{\varepsilon}} +\frac{1}{T} \frac{1}{e^{\frac{\eta}{\varepsilon}} - e^{- \frac{\bar V + \eta }{\varepsilon}}} \Big )
=0,
\end{align*}
which concludes the proof of the first statement.

\no The second statement follows taking the closed set
\begin{align*}
F := \Big \{ z \in D^c ~|~  |X^\varepsilon_{\sigma^\varepsilon(x)} - z^{*}| < \delta  \Big  \}.
\end{align*}
\no The previous proof implies immediately that
\begin{align*}
\displaystyle \lim_{\varepsilon \ra 0} \bar \PP \Big ( |X^\varepsilon_{\sigma^\varepsilon(x)} - z^{*}| < \delta \Big )=1.
\end{align*}
 \end{enumerate}
 \end{proof}
\section{Appendix: the large deviations regime for a pure-jump process}

\no In this subsection we prove that rescaling in a inverse way the size and the intensity of the jumps of a pure jump process produces a  large deviations principle in the small noise limit.  For any $\varepsilon>0$ and $T>0$ let 
\begin{align*}
L^\varepsilon_t := \int_0^t \int_{\RR^d \backslash \{0\}} z N^{\frac{1}{\varepsilon}}(ds,dz), \quad t \in [0,T]
\end{align*}
be the pure jump process driven by the Poisson random measure $N^{\frac{1}{\varepsilon}}$ defined on $(\Omega, \fF, \PP)$ with intensity measure $\frac{1}{\varepsilon} \nu \otimes ds$ where $ds$ stands for the Lebesgue measure on the positive real line and $\nu$ is the L\'{e}vy measure on $(\RR^d \backslash \{0\}, \bB(\RR^d \backslash \{0\}))$ given by
\begin{align*}
\nu(dz)= \frac{1}{|z|^{d+\beta}} e^{- |z|^\alpha} dz, \quad \alpha > 1, \beta \in [0,1).
\end{align*}
The following proposition implies (by Theorem 1.3.7 in \cite{Dupuis Ellis}) that there exists a subsequence $\varepsilon(n) \ra 0$ as $n \ra \infty$ such that the family $(L^{\varepsilon(n)})_{\varepsilon(n)>0}$ satisfies a large deviations principle along for some good rate function.
\begin{proposition} \label{theorem: LDP pure jump process}
For any $M>0$ we have
\begin{align} \label{eq: exp tightness of Lepsilon}
\displaystyle \lim_{\varepsilon \ra 0} \varepsilon \ln \PP \Big (  L^\varepsilon_T > M \Big )= - \infty.
\end{align}
\end{proposition}
\begin{proof}
\no  For every $\varepsilon>0$ and $T>0$ we write
 \begin{align*}
 L^\varepsilon_T &=I^\varepsilon_T + J^\varepsilon_T \\
 & := \int_0^T \int_{0 < |z|< 1}  \varepsilon z N^{\frac{1}{\varepsilon}}(ds,dz) + \int_0^T \int_{|z| \geq 1} \varepsilon z N^{\frac{1}{\varepsilon}}(ds,dz).
 \end{align*}
 \no The stochastic process $(J^\varepsilon_t)_{t \in [0,T]}$ that reads as the big jumps of $(L^\varepsilon_t)_{t \in [0,T]}$ is a compound Poisson process. We represent
 \begin{align*}
 J^\varepsilon_T = \varepsilon \sum_{k=1}^\infty |W_k| \textbf{1}_{\{ T^\varepsilon_k \leq T\}}
 \end{align*}
 where the sequence of jumps $(W_k)_{k \in \NN}$ is i.i.d. with law $\frac{\nu_1}{\beta}$ where we define $\nu_1(A):= \nu(A \cap |z| \geq 1)$ and $\beta=\nu(B_1^c)< \infty$. The jump times $(T^\varepsilon_k)_{k \in \NN}$ are defined recursively as
\begin{align*}
 & T^\varepsilon_0:=0 \\
 & T^\varepsilon_k := \inf \{ s > T^\varepsilon_{k-1} ~|~ |\Delta_s L^\varepsilon| >1 \} \quad k \geq 1.
 \end{align*}

 \no The waiting times $(\tau^\varepsilon_k)_{k \in \NN}$ are defined for any $\varepsilon>0$ and $k \in \NN$ by
 \begin{align*}
 \tau^\varepsilon_k := T^\varepsilon_k - T^\varepsilon_{k-1} \sim EXP \Big ( \frac{\beta}{\varepsilon} \Big ).
 \end{align*}
 For every $\varepsilon>0$ let $(N^\varepsilon_t)_{t \in [0,T]}$ be the Poisson clock with intensity $\frac{\beta}{\varepsilon}$.
 \no For every $\varepsilon>0$, $T>0$ and  $M>0$ we have
 \begin{align} \label{eq: estimate big jumps}
 \PP \Big ( J^\varepsilon_T >M \Big) &= \PP \Big  ( \varepsilon \sum_{k=1}^\infty |W_k| \textbf{1}_{\{T^\varepsilon_k \leq T \}}>M \Big ) \nonumber \\
 & = \PP \Big ( \sum_{k=1}^{N^\varepsilon_T} |W_k| > \frac{M}{\varepsilon} \Big ) \nonumber \\
 &= \sum_{k=1}^{\infty} \PP \Big ( \sum_{k=1}^n |W_k| > \frac{M}{\varepsilon} \Big | N^\varepsilon_T =n  \Big ) \PP(N^\varepsilon_T=n) \nonumber \\
 & \leq \sum_{n=1}^\infty \sum_{k=1}^n \PP \Big ( |W_k| > \frac{M}{\varepsilon} \Big ) e^{- \frac{\beta}{\varepsilon}T} \frac{1}{n!} \Big ( \frac{\beta T}{n} \Big )^n.
 \end{align}
 \no We estimate the deviation of one single jump, for any $u \geq 0$, by
 \begin{align*}
 \PP \Big ( |W_1| > u \Big ) &= \frac{1}{\beta} \int_{B^c_u} \nu_1(dz) \leq \frac{1}{\beta} \int_{B^c_u \cap |z| \geq 1} e^{- |z|^\alpha} dz \\
 & \leq c_d \int_u^\infty e^{- x^\alpha} x^{d-1} \\
 & \leq \frac{c_d}{\alpha} \Gamma \Big ( \frac{d}{\alpha}, u^\alpha \Big )
 \end{align*}
for some normalizing constant $c_d>0$ and where the Euler's Gamma function $\Gamma$ is defined by
\begin{align*}
\Gamma(s,y)= \int_y^{\infty} z^{s-1} e^{-z} ds, \quad s,y \in \RR.
\end{align*}
\no We use the well-known asymptotic behaviour of the function $\Gamma$
\begin{align*}
\displaystyle \lim_{y \in \infty} \frac{\Gamma(s,y)}{y^{s-1}e^{-y}}=1
\end{align*}
and we conclude for some $C>0$ and every $u>0$ that
\begin{align} \label{eq: law one big jump}
\PP \Big ( |W_1| >u \Big ) \leq \frac{C}{\beta} u^{d - \alpha} e^{- \frac{u^\alpha}{2}}. 
\end{align}
\no Combining (\ref{eq: estimate big jumps}) and (\ref{eq: law one big jump}) yields for any $M>0$ that
\begin{align*}
\PP \Big ( J^\varepsilon_T > M \Big ) \leq \frac{C}{\beta} \sum_{n=1}^\infty \frac{1}{(n-1)!} \Big ( \frac{M}{\varepsilon} \Big )^{d - \alpha} e^{- \frac{1}{2} \Big ( \frac{M}{\varepsilon} \Big )^\alpha} e^{- \frac{\beta}{\varepsilon} T} \Big ( \frac{\beta T}{\varepsilon}\Big )^n.
\end{align*}
\no Therefore for every $\varepsilon>0$ we have at a logarithmic scale that
\begin{align} \label{eq: the LDP big jumps scale}
\varepsilon \ln \PP \Big ( J^\varepsilon_T > M  \Big ) \lesssim_\varepsilon - \frac{M^\alpha}{\varepsilon^{\alpha-1}} - \beta T.
\end{align}
\no We proceed now to examine the logarithmic deviations of the small jumps family $(I^\varepsilon_T)_{\varepsilon>0}$. For any $\lambda=\lambda(\varepsilon)>0$ such that $\lambda(\varepsilon) \varepsilon^2 \ra 0$ as $\varepsilon \ra 0$ we have , due to Campbell's inequality,  for any $\varepsilon>0$ and $M>0$ that
\begin{align} \label{eq: the small jumps estimate}
\PP \Big ( I^\varepsilon_T > M \Big ) & = \PP \Big ( e^{\lambda (I^\varepsilon_T)^2 } > e^{\lambda M^2} \Big ) \nonumber \\
& \leq e^{- \lambda M^2} \EE \ \Big [e^{\lambda (I^\varepsilon_T)^2 } \Big ] \nonumber \\
& \leq e^{- \lambda M^2} \exp \Big ( \frac{T}{\varepsilon} \int_{\RR^d \backslash \{0\}} (e^{\lambda \varepsilon^2 |z|^2} -1) \nu(dz) \Big ) \nonumber \\
& \leq e^{- \lambda M^2} \exp \Big ( \frac{T}{\varepsilon} \lambda \varepsilon^2 \int_{0 < |z|<1} |z|^2 \nu(dz)\Big ).
\end{align}
In the last estimate we used the fact that 
\begin{align*}
\frac{e^{\lambda \varepsilon^2 |z|^2}-1}{\lambda \varepsilon^2} \ra 1 
\end{align*}
since $\lambda \varepsilon^2 \ra 0$ as $\varepsilon \ra 0$.  We fix the parametrization $\lambda(\varepsilon)= \frac{1}{\varepsilon^{1+p}}$ for some $p \in (0,1)$ for this purpose.

\no The estimate (\ref{eq: the small jumps estimate}) yields for every $M>0$ that
\begin{align} \label{eq: the LDP small jumps scale}
\varepsilon \ln \PP \Big ( I^\varepsilon_T > M \Big ) \lesssim_{\varepsilon} - \lambda(\varepsilon) \varepsilon M^2 + \lambda(\varepsilon) \varepsilon^2 \ra 0 \text{ as } \varepsilon \ra 0
\end{align}
due to the choice of $\lambda=\lambda(\varepsilon)>0$ fixed above.

\no Hence combining (\ref{eq: the LDP big jumps scale}) and (\ref{eq: the small jumps estimate}) we conclude the desired limit (\ref{eq: exp tightness of Lepsilon}).
\end{proof}

\paragraph*{Acknowledgments.} The author acknowledge and thanks the financial support from the FAPESP grant number 2018/06531-1 at the University of Campinas (UNICAMP) SP-Brazil,  the FAPESP grant number 2019/21324-5 at ENSTA-ParisTech, Palaiseau-France and the FAPESP grant number 2020/04426-6.

\end{document}